\documentclass[twoside,reqno]{amsart}
\usepackage{pgfplots}

\usepackage{amsfonts,amsmath,amscd,amsthm,amssymb}
\usepackage{mathtools}
\usepackage{url}
\usepackage[mathscr]{euscript} 
\usepackage{graphics}
\usepackage{wrapfig}
\usepackage{lscape}
\usepackage{rotating}
\usepackage{epsfig}
\usepackage{cite}
\usepackage{color}
\usepackage{booktabs}
\usepackage[vlined,ruled]{algorithm2e}
\usepackage{multirow}

\usepackage{tikz}
\usetikzlibrary{patterns}

\newcommand\gfu[1]{\textcolor{black}{#1}}
\newcommand\gfv[1]{\textcolor{black}{#1}}

\newtheorem{theorem}{Theorem}

\newtheorem{remark}{Remark}

\def\restrict#1{\raise-0.2ex\hbox{\ensuremath|}_{#1}}

\newcommand{\bld}[1]{\boldsymbol{#1}}

\newcommand{\Sh}{\bld{V}_{\!h}}
\newcommand{\Shd}{\bld{V}_{\!h,\mathsf{dg}}}

\newcommand{\divs}{{\nabla\cdot}}
\newcommand{\grads}{{\nabla}}

\newcommand{\pol}{\mathbb{P}}

\newcommand{\Oh}{{\mathcal{T}_h}}

\newcommand{\Eh}{\mathcal{F}_h}

\newcommand{\jmp}[1]{[\![#1 ]\!]}
\newcommand{\vertiii}[1]{{\left\vert\kern-0.25ex\left\vert\kern-0.25ex\left\vert #1
    \right\vert\kern-0.25ex\right\vert\kern-0.25ex\right\vert}}

\begin{document}
\title[Explicit divergence-free DG]{
An explicit divergence-free DG method for incompressible flow}
\author{Guosheng Fu}
\address{Division of Applied Mathematics, Brown University, 182 George St,
Providence RI 02912, USA.}
\email{Guosheng\_Fu@brown.edu}


\begin{abstract}
We present an explicit divergence-free DG method for incompressible flow based 
on velocity formulation only. \gfu{An $H(\mathrm{div})$-conforming, and} globally divergence-free finite element space is used for the velocity field, 
and the pressure field is eliminated from the equations by design.
The resulting ODE system can be discretized using any explicit time stepping methods. We use the 
third order strong-stability preserving Runge-Kutta method in our numerical experiments.
Our spatial discretization produces the {\it identical} velocity field as 
the divergence-conforming DG method of Cockburn et al. \cite{CockburnKanschatSchotzau07} based on 
a velocity-pressure formulation,
when the same DG operators are used for the convective and viscous parts.

Due to the global nature of the divergence-free constraint
\gfv{and its interplay with the boundary conditions, 
it is very hard to construct
local bases for our finite element space}. 
\gfv{Here} we present a key result on the efficient implementation of the scheme 
by identifying the equivalence of the  mass matrix inversion of 
the globally divergence-free finite element
space to a standard (hybrid-)mixed Poisson solver. 
Hence, in each time step, a (hybrid-)mixed Poisson solver is used, which reflects the 
global nature of the 
incompressibility condition. 
In the actual implementation of this fully discrete scheme, the pressure field is 
also computed (via the hybrid-mixed Poisson solver). Hence, the scheme can be interpreted
as a velocity-pressure formulation that treat the incompressibility constraint 
\gfu{and pressure forces} implicitly, 
but the viscous and convective part explicitly.
Since we treat viscosity explicitly for the Navier-Stokes equation, our method shall be 
best suited for unsteady high-Reynolds number flows 
so that the CFL constraint is not too restrictive.
\end{abstract}
\maketitle

\section{Introduction}
\label{sec:intro}
It is highly desirable to have 
a velocity field that is point-wisely divergence-free (exactly mass conservation) for incompressible flows; see the recent review
article \cite{John17}. 

We propose a new explicit, high-order, divergence-free DG scheme for the unsteady 
incompressible Euler and Navier-Stokes equation based on a {\it solely} velocity formulation.
The pressure field and incompressibility constraint are eliminated from the equation by design.
Our semi-discrete scheme produce exactly the same velocity field 
as the divergence-conforming DG method of 
Cockburn et al. \cite{CockburnKanschatSchotzau07}. 
Hence, our scheme enjoys features such as global and local conservation properties, 
high-order accuracy, energy-stability, \gfv{and} pressure-robustness \cite{CockburnKanschatSchotzau07,Guzman17}.

The resulting semi-discrete scheme is an ODE system for velocity only, 
as opposite to the differential-algebraic equations (DAE) in \cite{CockburnKanschatSchotzau07} 
where the pressure field and incompressibility-constraint enter into the equations directly.
As a consequence, we can apply any explicit time-stepping techniques to solve the ODE system.
\gfu{Our explicit fully-discrete scheme is also equivalence to the velocity-pressure formulation
\cite{CockburnKanschatSchotzau07} coupled with corresponding 
explicit treatments for the convective and viscous parts, and
implicit treatments for  the pressure forces and divergence-free constraint.
Such temporal treatment has already been briefly discussed in \cite[Section 3.2.1]{Lehrenfeld:10}.}

Within each time step, the mass matrix for the divergence-free finite element space 
shall be inverted.
\gfv{Due to the non-locality of the divergence-free constraint in the finite element space and 
its interplay with the boundary conditions, 
it is very hard, if possible,  to construct the local bases.
Here we consider alternative formulations for the efficient implementation of the 
fully-discrete scheme.
In particular, we either relax the 
divergence-free condition or the divergence-conformity condition
in the finite elements via proper Lagrange multipliers, which yields a mixed Poisson solver or a 
hybrid-mixed Poisson solver in each time stage.
The hybrid-mixed formulation is used in our numerical simulations.
}

We treat the viscosity term explicitly to avoid a Stokes solver. Hence, our scheme shall be 
applied to unsteady high-Reynolds number, \gfu{unresolved}  flows so that the CFL constraint is not 
too restrictive.
\gfu{Roughly speaking, when both convective and viscous terms are treated explicitly 
as in our scheme, 
the following time stepping restriction for stability is to be expected
\begin{align*}
 \Delta t \le \min\left\{c_C\frac{h}{k^2}\frac{1}{v_{\max}}, c_B\frac{h^2}{k^4}\frac{1}{\nu}\right\}, 
\end{align*}
where $\Delta t$ is the time step size, $h$ is the mesh size, $k$ is 
the polynomial degree in the finite elements, $v_{\max}$ is the maximal velocity magnitude, 
$\nu$ is the viscosity coefficient, and 
$c_B, c_C>0$ are the CFL stability constants for the convective and viscous parts, respectively.
If we denote the mesh Reynolds number $\mathrm{Re}_{h}$ as
\begin{align}
\label{reh}
  \mathrm{Re}_{h}: = \frac{v_{\max}h}{\nu\,k^2},
\end{align}
then the above time stepping restriction  becomes 
\begin{align}
\label{ext}
  \Delta t \le \min\left\{c_C, c_B  \mathrm{Re}_{h}\right\} \frac{h}{k^2}\frac{1}{v_{\max}}.
\end{align}
Hence, as long as the mesh Reynolds number $\mathrm{Re}_{h}\gg 1$ (unresolved flow), or 
$c_C \approx c_B\, \mathrm{Re}_{h}$ (slightly resolved flow), 
the explicit treatment of viscous term does not pose extra severe time-stepping 
restrictions besides the CFL constraint from the explicit convection treatment.}

\gfu{
On the other hand, when $\mathrm{Re}_{h}\ll 1$, i.e., when the flow is highly resolved, 
explicit treatment of the viscous term would not be efficient anymore.
In this case, we suggest to treat the viscous term implicitly
with a divergence-conforming hybridizable DG 
(HDG) method \cite{Lehrenfeld:10,LehrenfeldSchoberl16}.
Therein, various {\it stiffly accurate} operator-splitting 
time integration approaches were discussed, including  
{\it additive} decomposition methods like
IMplicit-EXplicit(IMEX) schemes
\cite{AscherRuuthSpiteri97,CalvoFrutosNovo01,KennedyCarpenter03}, 
{\it product} decomposition methods like the operator-integration-factor 
splittings \cite{Maday90}, and an operator-splitting modification of the 
fractional step method \cite{Glowinski03}.
}


Comparing with other schemes that treat viscosity explicitly,  
the computational cost of our scheme is comparable to the DG scheme based on a 
vorticity-stream function formulation \cite{LiuShu00} in two dimensions, and is 
a lot cheaper than the vorticity-vector potential formulation \cite{ELiu97} in three dimensions.
A significant computational saving per time step (one hybrid-mixed Poisson solver/step)
is achieved comparing with methods that treat viscosity term implicitly, e.g.
the IMEX divergence-conforming HDG scheme \cite{LehrenfeldSchoberl16} (one Stokes solver/step) or 
the projection methods \cite{Guermond06} ($d+1$ Poisson solver/step with $d$ the space dimension). 
Finally, we shall mention that boundary condition is easy to impose 
for our velocity-based formulation
(and for various mixed methods based on velocity-pressure formulations \cite{John17}),
while that consists one of the major bottlenecks for 
vorticity-based methods \cite{ELiu96} or projection methods \cite{Guermond06}.

The rest of the paper is organized as follows.
In Section 2, the explicit divergence-free DG scheme is introduced for the 
incompressible Euler equation, along with a key result on 
transforming the mass matrix inversion to a hybrid-mixed Poisson solver.
In Section 3, the scheme is extended to the incompressible Navier-Stokes equations.
Extensive numerical results in two dimensions are presented in Section 4.
Finally we conclude in Section 5.

\section{Euler equations}
\label{sec:model}
We consider the following incompressible Euler equations:
\begin{subequations}
\label{euler}
 \begin{align}
  \label{euler-1}
\partial_t \bld u +(\bld u\cdot\nabla)\bld u + \nabla p = &\; \bld f, && \text{ in }\Omega,
  \\
  \label{euler-2}
\divs \bld u = &\; 0, && \text{ in }\Omega,
  \\
  \label{euler-3}
 \bld u\cdot \bld n = &\; g, && \text{ on }\partial\Omega,
 \end{align}
\end{subequations}
with initial condition 
\[
\bld u(x,0) = \bld u_0(x)\quad \forall x\in \Omega,
\]
where $\bld u$ is the velocity and $p$ is the pressure, $\Omega\subset\mathbb{R}^d$(d=2,3) is a polygonal/polyhedral domain, 
and $\bld n$ is the outward normal direction on the domain boundary $\partial\Omega$.
The initial velocity $\bld u_0(x)$ is assumed to be divergence-free.
For simplicity, we assume no source/sink and no-flow boundary conditions, $\bld f=0$ and $g=0$.
The inflow/outflow boundary conditions
will be discussed at the end of this section.

\subsection{Preliminaries}
Let $\Oh$ be a conforming simplicial triangulation of $\Omega$.
For any element $T \in\Oh$, we denote by $h_T$ its diameter and 
we denote by $h$ the maximum diameter over all mesh elements. 
Denote by $\Eh$ the set of facets of  $\Oh$, and by $\Eh^i=\Eh\backslash\partial\Omega$
the set of interior facets. 

\def\jump#1{[\![{#1}]\!]} 
\def\mean#1{\{\!\!\{{#1}\}\!\!\}} 

We denote the following set of finite element spaces:
\begin{subequations}
\label{space}
\begin{align}
\label{space-1}
\Shd^k : =&\; \prod_{T\in\Oh} [\pol^{k}(T)]^d,\\
\label{space-2}
\Shd^{k,m} : =&\; \{\bld v\in \Shd^k, \;\;
\divs \bld v|_{T}\in \pol^m(T)\;\;\forall T\in\Oh.\},\\
\label{space-3}
\Sh^k : =&\; \{\bld v \in \Shd^k, \;\;
\jump{\bld v\cdot\bld n}_F = 0 \;\;\forall F\in\Eh.\}\subset H_0(\mathrm{div},\Omega),\\
\label{space-4}
\Sh^{k,m} : =&\; \{\bld v\in \Sh^k, \;\;
\divs \bld v\in \pol^m(T)\;\;\forall T\in\Oh.\},\\
\label{space-5}
Q_h^m : =&\; \left(\prod_{T\in\Oh} \pol^{m}(T)\right)\cap L_0^2(\Omega),\\
\label{space-6}
M_h^k : =&\; \prod_{F\in\Eh} \pol^{k}(F),
\end{align}
where the polynomial degree $k\ge 1$ and $-1\le m\le k-1$, and
$\jmp{\cdot}$ is the usual jump operator and $\pol^r$ the space of 
polynomials up to degree $r$ with the convention that $\pol^{-1}=\{0\}$.
Note that functions in $M_h^k$ are defined only on the mesh skeleton $\Eh$, which will be used in 
the hybrid-mixed Poisson solver.
\end{subequations}

Finally, we introduce the jump and average notation. 
Let $\bld \phi_h$ be any function in $\Shd^k$.
On each facet $F\in\Eh^i$ shared by two elements $K^-$ and $K^+$, we denote 
$(\bld \phi_h)^\pm|_F =\left.(\bld\phi_h)\right|_{K^\pm}$, and use
\begin{align}
 \label{avg-jmp2d}
\jump {\bld \phi_h}|_F  = \bld\phi_h^+\cdot\bld n^++ \bld\phi_h^-\cdot\bld n^-,\quad \quad
\mean {\bld \phi_h}|_F  = \frac12(\bld \phi_h^++ \bld \phi_h^-)
\end{align}
to denote the jump and the average of $\phi_h\in V_h^k$ on the facet $F$.

\subsection{Spatial discretization}
The divergence-free space $\Sh^{k,-1}$ shall be used in our DG formulation.
With this space in use, the divergence-free constraint \eqref{euler-2} is
point-wisely satisfied by design, and the pressure do not enter into the weak formulation of the 
scheme.
The semi-discrete scheme reads as follows:
find $\bld u_h(t)\in \Sh^{k,-1}$
such that
 \begin{align}
  \label{scheme-euler}
(\partial_t \bld u_h, \bld v_h)_{\Oh} + \mathcal{C}_h(\bld u_h; \bld u_h, \bld v_h) = 0,
\quad \forall \bld v_h\in \Sh^{k,-1}.
\end{align}
where $(\cdot,\cdot)_\Oh$ denotes the standard $L^2$-inner product,
and the {\it upwinding} trilinear form 
 \begin{align}
  \label{adv}
\mathcal{C}_h(\bld u_h; \bld u_h, \bld v_h):=
\sum_{T\in\Oh}\int_T -(\bld u_h\otimes\bld u_h):\grads \bld v_h\,\mathrm{dx}
  +\int_{\partial T}(\bld u_h\cdot\bld n)({\bld u}_h^{-}\cdot\bld v_h)\,\mathrm{ds} \nonumber\\
 \end{align}
 where the upwinding numerical flux ${\bld u}_h^-|_F = \bld u_h|_{K^-}$
 with $K^-$ being the element
 such that its outward normal direction $\bld n^-$ on the facet $F$  satisfies 
 ${\bld u}_h^-\cdot \bld n^-\ge 0$ (outflow boundary).

 Since 
\[
 \mathcal{C}_h(\bld u_h; \bld u_h, \bld u_h)=
 \sum_{F\in\Eh^i}\int_{F}|\bld u_h\cdot\bld n|(\jump{\bld u_h}\cdot \jump{\bld u_h})\,\mathrm{ds}
 \ge 0,
\]
the scheme \eqref{scheme-euler} is energy-stable in the sense that 
\[
\partial_t \|\bld u_h^2(t)\|_{\Oh}\le 0,
\]
where $\|\cdot\|_\Oh$ denotes the $L^2$-norm on $\Oh$.

\subsection{Temporal discretization}
The semi-discrete scheme \eqref{scheme-euler} can be written as
\[
 \mathcal{M}(\partial_t\bld u_h) = \mathcal{L}(\bld u_h),
\] 
where $\mathcal{M}$ is the mass matrix for the space $\Sh^{k,-1}$, and 
$\mathcal{L}(\bld u_h)$ the spatial discretization  operator.
Any explicit time stepping techniques can be applied to the scheme \eqref{scheme-euler}.
We use the following three-stage, third-order strong-stability 
preserving Runge-Kutta method (TVD-RK3) \cite{ShuOsher88} in 
our numerical experiments: 
\begin{align}
\label{rk3}
 \mathcal{M}\bld u_h^{(1)} = &\; \mathcal{M}\bld u_h^n +\Delta t^n \mathcal{L}(\bld u_h^n),\nonumber\\
\mathcal{M}\bld u_h^{(2)} = &\; \frac34\mathcal{M}\bld u_h^n +
\frac14\left[\mathcal{M}\bld u_h^{(1)}+\Delta t^n \mathcal{L}(\bld u_h^{(1)})\right],\\
\mathcal{M}\bld u_h^{n+1} = &\; \frac13\mathcal{M}\bld u_h^n 
+\frac23\left[\mathcal{M}\bld u_h^{(2)}+
\Delta t^n \mathcal{L}(\bld u_h^{(2)})\right],\nonumber
\end{align}
where $\bld u_h^{n}$ is the given velocity at time level $t^n$ and 
$\bld u_h^{n+1}$ is the computed velocity at time level $t^{n+1} = t^n+\Delta t^n$.
In each time step, three mass matrix inversion is needed. 

\begin{remark}[Implementation]
Despite the mathematical simplicity of the solely velocity based formulation \eqref{scheme-euler} 
and the ease of using explicit time stepping methods of the resulting ODE system,
to the best of our knowledge, the method was never directly implemented in the literature.
The major obstacle is that the space $\Sh^{k,-1}$ is not a
standard finite element space
\gfv{whose basis functions can be easily defined},
due to the built-in global 
divergence-free constraint \gfv{and its interplay with the boundary conditions}. 
\gfv{By the finite element de Rham complex property\cite{ArnoldFalkWinther06}, we have, in two dimensions,
\begin{align*}
 \Sh^{k,-1} = \nabla \times \Big\{\phi\in H^1(\Omega):&\quad \phi|_T\in \pol^{k+1}(T),\;\;\;\forall T\in\Oh, 
 \\
&\;\;\quad\quad (\nabla\times \phi)\cdot \bld n=0 \text{ on }\partial\Omega\Big\},
\end{align*}
where the two-dimension $\mathrm{curl}$ operator `` $\nabla\times$''  is the rotated gradient,
and, in three dimensions, 
\begin{align*}
 \Sh^{k,-1} = \nabla \times \Big\{\bld \phi\in H(\mathrm{curl},\Omega):&\quad \bld \phi|_T\in 
 [\pol^{k+1}(T)]^3,
 \;\forall T\in\Oh, \\
&\;\;\;\quad\quad (\nabla\times \bld \phi)\cdot \bld n=0  \text{ on }\partial\Omega\Big\}.
\end{align*}
The difficulty of basis construction of this space in two dimensions lies in the 
treatment of the boundary condition for domain with more than one piece of connected boundary, 
which can be resolved by a weakly enforcement of boundary conditions.
On the other hand, the difficulty of basis construction 
in three dimensions is more fundamental, which is due to the fact that 
the $\mathrm{curl}$ operator has a large kernel including all gradient fields. 
}

\gfv{
In the next subsection, we introduce proper Lagrange multipliers to avoid the direct use of the 
divergence-free space $\Sh^{k,-1}$. 
}
\end{remark}

\subsection{\gfv{Avoid bases construction for $\Sh^{k,-1}$}}
\label{sec:ref}
In this subsection, we show an efficient 
implementation of the scheme coupled with forward Euler time stepping
that \gfv{\it avoid bases construction for the space $\Sh^{k,-1}$}.
The forward Euler scheme for \eqref{scheme-euler} reads as follows:
given the numerical solution at time $t^n$, $\bld u_h^n\in\Sh^{k,-1}\approx \bld u(t^n)$, 
compute the solution at next time level
$\bld u_h^{n+1}\in\Sh^{k,-1}\approx \bld u(t^{n}+\Delta t^n)$ by the following set of equations:
\begin{align}
 \label{FE}
 (\bld u_h^{n+1}, \bld v_h)_{\Oh} =
\underbrace{(\bld u_h^{n}, \bld v_h)_{\Oh}- \Delta t^n\, \mathcal{C}_h(\bld u_h^n; \bld u_h^n, \bld v_h)}_{:=\mathcal{F}^n(\bld v_h)},
  \quad \forall \bld v_h\in\Sh^{k,-1}.
\end{align}

\subsubsection*{\gfu{Conversion} to a mixed-Poisson solver (via a velocity-pressure formulation)}
We introduce the following equivalent formulation of the scheme \eqref{FE} that 
use a larger velocity space that is divergence-conforming, but not divergence-free:
find $({\bld u}_{h,\mathrm{mix}}^{n+1}, w_h^{n+1})\in \Sh^{k}\times Q_h^{k-1}$ such that 
\begin{subequations}
 \label{FE2}
\begin{alignat}{2}
 \label{FE2-1}
 (\bld u_{h,\mathrm{mix}}^{n+1}, \bld v_h)_{\Oh}
 -
 (w_{h}^{n+1}, \divs \bld v_h)_{\Oh}
 =&\;\mathcal{F}^n(\bld v_h),
&&\;  \quad \forall \bld v_h\in\Sh^{k},\\
 \label{FE2-2}
 (\divs \bld u_{h,\mathrm{mix}}^{n+1}, z_{h})_{\Oh} =&\; 0 
&&\;  \quad \forall z_h\in Q_h^{k-1}.
\end{alignat}
\end{subequations}
Since $\Sh^k$ is the standard BDM space \cite{BrezziDouglasMarini85}, the scheme \eqref{FE2} can be readily implemented.
Notice that the scheme \eqref{FE2} is nothing but a mixed Poisson formulation (with a different 
right hand side vector), 
whose well-posedness is well-known.

The equivalence of schemes \eqref{FE} and \eqref{FE2} is given below.
\begin{theorem}
 Let $({\bld u}_{h,\mathrm{mix}}^{n+1}, w_h^{n+1})\in \Sh^{k}\times Q_h^{k-1}$ 
 be the unique solution to the equations \eqref{FE2}. Then, ${\bld u}_{h,\mathrm{mix}}^{n+1}\in \Sh^{k,-1}$
 solves the equations \eqref{FE}.
 Moreover, the quantity $w_h^{n+1}/\Delta t^n$ is an approximation of the pressure field at time $t^{n+1}$.
\end{theorem}
\begin{proof}
 Since $\divs\Sh^{k} = Q_h^{k-1}$, the equations \eqref{FE2-2} implies that 
$\divs \bld u_{h,\mathrm{mix}}^{n+1}=0$, hence $\bld u_{h,\mathrm{mix}}^{n+1}\in\Sh^{k,-1}$. 
Taking $\bld v_h\in \Sh^{k,-1}$ in equations \eqref{FE2-1}, we get
\[(\bld u_{h,\mathrm{mix}}^{n+1}, \bld v_h)_{\Oh}=\;\mathcal{F}^n(\bld v_h).\]
Hence, $\bld u_{h,\mathrm{mix}}^{n+1}\in \Sh^{k,-1}$ is the solution to equations \eqref{FE}.
The quantity $w_h^{n+1}/\Delta t^n$ approximates the pressure variable is shown in 
Remark \ref{rk1} below.
\end{proof}

\begin{remark}[Equivalence with the velocity-pressure formulation]
\label{rk1}
Recall that the semi-discrete velocity-pressure formulation \cite{Guzman17}
which use a divergence-conforming velocity space $\Sh^k$ and 
the matching pressure space $Q_h^{k-1}$ is to find $(\bld u_h(t), p_h(t))\in \Sh^k\times Q_h^{k-1}$
such that
\begin{align*}
(\partial_t \bld u_h, \bld v_h)_{\Oh} -(p_h,\divs\bld v_h)_\Oh+ \mathcal{C}_h(\bld u_h; \bld u_h, \bld v_h) =&\; 0,
\quad \forall \bld v_h\in \Sh^{k},\\
(\divs\bld u_h, q_h)_\Oh =&\; 0,
\quad \forall q_h\in Q_h^{k-1}.
\end{align*}
An first-order IMEX time discretization yields the fully-discrete scheme
\begin{subequations}
\label{vp}
\begin{alignat}{2}
(\bld u_h^{n+1}, \bld v_h)_{\Oh} -(\Delta t^n p_h^{n+1},\divs\bld v_h)_\Oh
=&\; \mathcal{F}^n(\bld v_h),
&&\quad \forall \bld v_h\in \Sh^{k},\\
(\divs\bld u_h^{n+1}, q_h)_\Oh =&\; 0,
&&\quad \forall q_h\in Q_h^{k-1},
\end{alignat}
which is easily seen to be identical to the scheme \eqref{FE2} introduced earlier by
identifying $w_h^{n+1}$ with $\Delta t^n p_h^{n+1}$.
Hence, although we work with the velocity-only formulation \eqref{scheme-euler} mathematically, 
in the numerical implementation, pressure is always simultaneously been calculated.

However, we point out that the formulation \eqref{FE2} is preferred over \eqref{vp}
in the actual numerical implementation due to the fact that the matrix resulting from 
the linear system never changes for variable time step size $\Delta t^n$, which is equivalent to 
a mixed Poisson solver.
\end{subequations}
\end{remark}

Although efficient solvers are available for the mixed-Poisson 
saddle point linear system \eqref{FE2}, 
we prefer to use the celebrated hybridization technique \cite{ArnoldBrezzi85}
to convert it to a symmetric positive definition linear system, which is a lot easier to 
solve.
\subsubsection*{\gfu{Convertion} to a hybrid-mixed Poisson solver}
The hybrid-mixed formulation is given below:
find $({\bld u}_{h,\mathrm{hyb}}^{n+1}, w_h^{n+1},\lambda_h^{n+1})\in \Shd^{k}\times Q_h^{k-1}\times 
M_h^k$ such that 
\begin{subequations}
 \label{FE3}
\begin{alignat}{2}
 \label{FE3-1}
 (\bld u_{h,\mathrm{mix}}^{n+1}, \bld v_h)_{\Oh}
 -
 (w_{h}^{n+1}, \divs \bld v_h)_{\Oh}&\nonumber\\
 + \sum_{T\in\Oh}\int_{\partial T} 
 \lambda_{h}^{n+1}(\bld v_h\cdot \bld n)\,\mathrm{ds}
 &=\;\mathcal{F}^n(\bld v_h),
&&\;  \quad \forall \bld v_h\in\Shd^{k},\\
 \label{FE3-2}
 (\divs \bld u_{h,\mathrm{mix}}^{n+1}, z_{h})_{\Oh} &=\; 0, 
&&\;  \quad \forall z_h\in Q_h^{k-1},\\
 \label{FE3-3}
\sum_{T\in\Oh}\int_{\partial T} 
 \mu_{h}(\bld u_h\cdot \bld n)\,\mathrm{ds} &=\; 0 
&&\;  \quad \forall \mu_h\in M_h^{k}.
\end{alignat}
\end{subequations}
After static condensation, the scheme \eqref{FE3} yields an SPD linear system for 
the Lagrange multiplier $\lambda_h^{n+1}$.
The following equivalence result is now trivially true.
\begin{theorem}
  Let $({\bld u}_{h,\mathrm{hyb}}^{n+1}, w_h^{n+1},\lambda_h^{n+1})\in \Shd^{k}\times Q_h^{k-1}\times 
  M_h^k$ 
 be the unique solution to the equations \eqref{FE3}. 
 Then, ${\bld u}_{h,\mathrm{hyb}}^{n+1}\in \Sh^{k,-1}$
 solves the equations \eqref{FE}.
  Moreover, the quantity $w_h^{n+1}/\Delta t^n$ is 
  an approximation of the pressure field at time $t^{n+1}$ on the mesh $\Oh$, while 
the quantity $\lambda_h^{n+1}/\Delta t^n$ is 
an approximation of the pressure field at time $t^{n+1}$ on the mesh skeleton $\Eh$.
\end{theorem}
\begin{remark}[More efficient implementation]
One can further improve the efficiency of this hybrid-mixed solver 
\eqref{FE3} by taking advantage of the divergence-free property of the velocity space.
In particular, we can restrict the velocity space to be {\it locally} divergence-free $\Shd^{k,-1}$, and 
remove the equations involving $w_h^{n+1}$ and $z_h$ in \eqref{FE3}, which results in the following 
simplified scheme: find $({\bld u}_{h,\mathrm{hyb}}^{n+1},\lambda_h^{n+1})\in \Shd^{k,-1}\times 
M_h^k$ such that 
\begin{subequations}
 \label{FE4}
\begin{alignat}{2}
 \label{FE4-1}
 (\bld u_{h,\mathrm{mix}}^{n+1}, \bld v_h)_{\Oh}
 - \sum_{T\in\Oh}\int_{\partial T} 
 \lambda_{h}^{n+1}(\bld v_h\cdot \bld n)\,\mathrm{ds}
 &=\;\mathcal{F}^n(\bld v_h),
&&\;  \quad \forall \bld v_h\in\Shd^{k,-1},\\
 \label{FE4-2}
\sum_{T\in\Oh}\int_{\partial T} 
 \mu_{h}(\bld u_h\cdot \bld n)\,\mathrm{ds} &=\; 0 
&&\;  \quad \forall \mu_h\in M_h^{k}.
\end{alignat}
\end{subequations}
Note that local bases for the DG space $\Shd^{k,-1}$ can be easily constructed.
\gfu{We mention that such spatial discretization which use 
a locally divergence-free velocity space and a hybrid (facet) pressure space was 
already considered in  \cite{Montlaur10}.
}

Our numerical simulations are performed using the open-source finite-element software 
{\sf NGSolve}\cite{Schoberl16}, \url{https://ngsolve.org/}, in which
we still use the formulation \eqref{FE3}, but take the velocity space
to be $\Shd^{k,0}$, and the space for $w_h^{n+1}$ to be piece-wise constants
$Q_h^0$.
The local bases for the space $\Shd^{k,0}$ for various element shapes can be found in 
\cite{Zaglmayr06}.
\end{remark}


\subsection{Inflow/outflow boundary conditions}
Finally, we briefly comment on the imposing of inflow/outflow boundary conditions.
Suppose the Euler equation \eqref{euler} is replaced with the following inflow/outflow/wall boundary
conditions:
\begin{alignat}{2}
\label{be}
 \bld u\cdot\bld n =&\; u_{\mathrm{in}} &&\;\; \text{ on } \Gamma_{in},\;\;\;\;
 \bld u\cdot\bld n =\; 0 \;\; \text{ on } \Gamma_{wall},\;\;\;\;
 p =\; 0 \;\; \text{ on } \Gamma_{out},
\end{alignat}
with $\partial \Omega =\Gamma_{in}\cup \Gamma_{wall}\cup \Gamma_{out}$.
Introducing the finite element spaces without/with boundary condition
\begin{align*}
 \widetilde{\Sh}^{k,-1} : =&\; \left\{\bld v \in \Shd^{k,-1}, \;\;
\jump{\bld v\cdot\bld n}_F = 0 \;\forall F\in\Eh^i,
\right\},\\
 \widetilde{\Sh}^{k,-1}_{,g} : =&\; \left\{\bld v \in  \widetilde{\Sh}^{k,-1}, \;\;
\bld v\cdot\bld n = \left\{\begin{tabular}{ll}
$g$ & on $\Gamma_{in}$,\\
$0$& on $\Gamma_{wall}.
$
                   \end{tabular}
\right.
\right\}.
\end{align*}
Note that for any pressure field $p\in H^1(\Omega)$ satisfying $p=0$ on $\Gamma_{out}$, the 
following identity holds, 
\begin{align}
\label{pres}
\int_\Omega \nabla p \cdot \bld v_h \mathrm{dx}
=&\;-\int_\Omega p \divs(\bld v_h) \mathrm{dx}
+\int_{\partial\Omega} p(\bld v_h\cdot\bld n) \mathrm{ds}\nonumber\\
=&\int_{\Gamma_{in}\cup\Gamma_{wall}} p(\bld v_h\cdot\bld n)\mathrm{ds}\;
,\quad \forall \bld v_h\in
\widetilde{\Sh}^{k,-1}.
\end{align}
In particular, $\int_\Omega \nabla p \cdot \bld v_h \mathrm{dx}=0$
for all $\bld v_h\in
\widetilde{\Sh}_{,0}^{k,-1}.$

Then, the semi-discrete scheme for \eqref{euler} with boundary condition \eqref{be}
is to find $\bld u_h(t)\in \widetilde{\Sh}_{,u_\mathrm{in}}^{k,-1}$
such that
 \begin{align*}
(\partial_t \bld u_h, \bld v_h)_{\Oh} + \mathcal{C}_h(\bld u_h; \bld u_h, \bld v_h)
= 0, \quad \forall \bld v_h\in \widetilde{\Sh}_{,0}^{k,-1}.
\end{align*}
A corresponding implementation of an explicit fully discrete scheme in the spirit 
of subsection \ref{sec:ref} can be obtained easily.


\section{Navier-Stokes equations}
Now, we consider \gfu{extending} the scheme \eqref{scheme-euler} to 
the following incompressible Navier-Stokes equations with free-slip boundary conditions:
\begin{subequations}
\label{ns}
 \begin{align}
  \label{ns-1}
\partial_t \bld u +(\bld u\cdot\nabla)\bld u + \nabla p-\nu\triangle \bld u = &\; \bld f, && \text{ in }\Omega,
  \\
  \label{ns-2}
\divs \bld u = &\; 0, && \text{ in }\Omega,
  \\
  \label{ns-3}
 \bld u\cdot \bld n = &\; 0, && \text{ on }\partial\Omega,
  \\
  \label{ns-4}
\nu((\nabla \bld u)\, \bld n)\times \bld n = &\; 0, && \text{ on }\partial\Omega,
 \end{align}
\end{subequations}
with a divergence-free initial condition 
\[
\bld u(x,0) = \bld u_0(x)\quad \forall x\in \Omega.
\]
Here $\nu=1/\mathrm{Re}$ is the viscosity.
Again, we point out that other boundary conditions such as 
inflow/outflow/wall boundary conditions can be easily applied.

We discretize the viscous term using symmetric interior penalty DG method \cite{ArnoldBrezziCockburnMarini02}.
The semi-discrete scheme reads as follows:
find $\bld u_h(t)\in \Sh^{k,-1}$
such that
 \begin{align}
  \label{scheme-ns}
(\partial_t \bld u_h, \bld v_h)_{\Oh} + \mathcal{C}_h(\bld u_h; \bld u_h, \bld v_h)
+\mathcal{B}_h(\bld u_h, \bld v_h)= 0, \quad \forall \bld v_h\in \Sh^{k,-1},
\end{align}
where the advective trilinear form $\mathcal{C}_h$ is given by \eqref{adv}, and the 
viscous bilinear form $\mathcal{B}_h$ is given below
\begin{align}
  \label{viscous}
  \mathcal{B}_h(\bld u_h, \bld v_h)
  :=&\;\sum_{T\in\Oh}\int_T\nu\nabla u:\nabla v\,\mathrm{dx}\nonumber\\
&  -\sum_{F\in\Eh^i}\int_F\nu\mean{\nabla \bld u_h}\jump{\bld v_h\otimes\bld n}\,\mathrm{ds}\nonumber\\
&  -\sum_{F\in\Eh^i}\int_F\nu\mean{\nabla\bld v_h}\jump{\bld u_h\otimes\bld n}\,\mathrm{ds}\nonumber\\
&  +\sum_{F\in\Eh^i}\int_F\nu\frac{\alpha k^2}{h}\jump{\bld u_h\otimes\bld n}
  \jump{\bld v_h\otimes\bld n}\,\mathrm{ds},
\end{align}
with $\alpha>0$ a sufficiently large stabilization constant.

For a fully discrete scheme, we use the same explicit stepping as the inviscid case. 
\gfu{As mentioned in the introduction, our explicit method shall be applied to high Reynolds number flows 
such that the mesh Reynolds number $Re_h$ defined in \eqref{reh} is not too small
to avoid severe time stepping restrictions.}

\section{Numerical results}
\label{sec:numerics}
In this section, we present several numerical results in two dimensions.
The numerical results are performed using the NGSolve software \cite{Schoberl16}.
The first four tests are performed on triangular meshes, \gfu{where} the last one \gfu{on} 
a rectangular mesh (with the obvious modification of the divergence-conforming space 
from BDM \cite{BrezziDouglasMarini85} to RT \cite{RaviartThomas77}).
For the viscous operator \eqref{viscous}, 
we take the stabilization parameter $\alpha$ to be $2$. 
We use the TVD-RK3 time stepping \eqref{rk3} with sufficiently small time 
step size \gfu{for all the tests, except for \textbf{Example 1b} where the 
classical fourth order Runge-Kutta method is also used to check the temporal accuracy.}
We use a pre-factored sparse-Cholesky 
factorization for the hybrid-mixed Poisson solver that is needed in each time step.

\subsection*{\gfu{Example 1a: Spatial accuracy test}}
This example is used to check the \gfu{spatial accuracy} of our schemes, 
both for the Euler equations \eqref{euler}
and for the Navier-Stokes equations \eqref{ns} with $\mathrm{Re}=100$.
Following  \cite{LiuShu00}, we take the domain to be $[0,2\pi]\times [0,2\pi]$ and 
use a periodic boundary condition.
The initial condition and source term are chosen such that 
the exact solution is  
\[
 u_1 = -\cos(x)\sin(y)\exp(-2t/\mathrm{Re}), 
 u_2 = \sin(x)\cos(y)\exp(-2t/\mathrm{Re}).
\]
The $L^2$-errors in velocity at $t=1$ on unstructured triangular meshes
are shown in Table \ref{table:error}.
It is clear to observe optimal $(k+1)$-th order of convergence for both cases.

\begin{table}[ht]
\caption{\textbf{Example 1a:} History of convergence of the $L^2$-velocity errors.} 
\centering 
{%
\begin{tabular}{ cc cc|cc}
\toprule
  & & \multicolumn{2}{c}{Euler}
          & \multicolumn{2}{|c}{Navier-Stokes}
\tabularnewline
$k$& $h$  & error & \gfu{eoc} & error & \gfu{eoc}
\tabularnewline
\midrule
\multirow{4}{2mm}{1} 
& 0.7854  &  2.339e-01  &  \gfu{--}   
&  2.234e-01  &  \gfu{--}
\\ 
& 0.3927  &  5.638e-02  &  2.05   
 &  5.195e-02  &  2.10 \\ 
&0.1963   &  1.446e-02  &  1.96    
 &  1.250e-02  &  2.06
\\
&0.0982  &  3.616e-03  &  2.00    
&  2.882e-03  &  2.12
\\ [1.5ex]
\multirow{4}{2mm}{2} 
 & 0.7854  &  2.411e-02  &  \gfu{--}  
 &  2.193e-02  &  \gfu{--} \\ 
 & 0.3927  &  2.491e-03  &  3.27   
   &  2.142e-03  &  3.36 
 \\ 
 & 0.1963  &  2.968e-04  &  3.07   
   &  2.488e-04  &  3.11 \\ 
 & 0.0982  &  3.514e-05  &  3.08   
  &  2.792e-05  &3.16
 \\[1.5ex]
\multirow{4}{2mm}{3}
 & 0.7854  &  1.495e-03  &  \gfu{--}   
  &  1.338e-03  &  \gfu{--}
\\ 
 & 0.3927  &  7.883e-05  &  4.25   
  &  6.876e-05  &  4.28 \\ 
 & 0.1963  &  4.969e-06  &  3.99  
  &  4.392e-06  &  3.97 
 \\ 
 & 0.0982  &  2.907e-07  &  4.10    
   &  2.701e-07  &  4.02\\ 
\bottomrule
\end{tabular}}
\label{table:error} 
\end{table}

\subsection*{\gfu{Example 1b: Temporal accuracy test}}
\gfu{This example is used to check the \gfu{temporal accuracy} of our schemes.
We consider the Navier-Stokes equation \eqref{ns} with $\nu=1/4000$ on
a periodic domain $\Omega=[0,2\pi]\times [0,2\pi]$ with 
the following exact solution
\[
 u_1 = \sin(6\pi t)\sin(y), \;
 u_2 = \sin(6\pi t)\sin(2x).
\]
}
\gfu{We use a $P^6$ scheme \eqref{scheme-ns} 
on a fixed triangular mesh with mesh size $h=2\pi/32$ to keep the spatial 
error small. 
For the time discretization, we consider either the TVD-RK3 scheme \eqref{rk3}, or the 
classical four-stage, fourth order Runge-Kutta (RK4) scheme.
The $L^2$-errors in velocity at $t=0.1$ with different time step size 
are shown in Table \ref{table:error2}.
It is clear to observe third order of convergence for TVD-RK3, and 
fourth order of convergence for RK4.}

\gfu{\begin{table}[ht]
\caption{\textbf{Example 1b:} History of convergence of the $L^2$-velocity errors.} 
\centering 
{%
\begin{tabular}{ cc cc|cc}
\toprule
&  & \multicolumn{2}{c}{TVD-RK3}
          & \multicolumn{2}{|c}{RK4}
\tabularnewline
& $\Delta t$  & error & \gfu{eoc} & error & \gfu{eoc}
\tabularnewline
\midrule
& 0.1/4  &  3.301e-04  &  --
&  1.688e-04  & {--}
\\ 
 & 0.1/8  &  3.654e-05  &  3.18
 &  1.098e-05  &  3.94 \\ 
  & 0.1/16  &  4.459e-06  &  3.03
  &  6.998e-07  &  3.97
\\
 & 0.1/32  &  5.580e-07  &  3.00 
 &  4.414e-08  &  3.99 
\\ 
\bottomrule
\end{tabular}}
\label{table:error2} 
\end{table}}
\subsection*{Example 2: Double shear layer problem}
We consider the double shear layer problem used in \cite{Bell87,LiuShu00}. 
The Euler equation \eqref{euler}
on the domain $[0,2\pi]\times [0,2\pi]$ with a periodic boundary condition and an initial condition:
\begin{align}
 u_1(x,y,0) = &\;\left\{
 \begin{tabular}{ll}
$\mathrm{tanh}((y-\pi/2)/\rho)$  & $y\le \pi$\\[1ex]
$\mathrm{tanh}((3\pi/2-y)/\rho)$  & $y> \pi$\\
  \end{tabular}
 \right.,\\
u_2(x,y,0) = &\; \delta \sin(x), 
\end{align}
with $\rho = \pi/15$ and $\delta = 0.05$.

We use $P^3$ approximation on fixed uniform unstructured triangular meshes with mesh size $2\pi/40$ and $2\pi/80$, 
see Fig.~\ref{fig:ds}, and 
run the simulation  up to time $t=8$.
We plot the time history of total energy (square of the $L^2$-norm of velocity $\bld u_h$)
and total enstrophy (square of $L^2$-norm of vorticity $\omega_h:=\nabla_h\times \bld u_h$) 
in Fig.~\ref{fig:ds0},
as well as contours of the vorticity at $t=6$ and $t=8$ in 
Fig.~\ref{fig:ds1}
to show the resolution.
We can see from Fig.~\ref{fig:ds0} that the energy is monotonically decreasing, with a 
very small dissipation error. The dissipated energy at time $t=8$ for the 
scheme on the coarse mesh is about $2\times 10^{-3}$, 
while that on the fine mesh is about $2\times 10^{-4}$.
The dissipation in enstrophy is more severe, where we also 
observe a fluctuation, probably due to 
the fact that vorticity $\omega_h$ is a derived variable from the 
velocity computation. Our results are qualitatively similar to those obtained in \cite{LiuShu00}
that use a vorticity-stream function formulation, with roughly a similar computational cost.

\begin{figure}[ht!]
 \caption{Example 2: the computational mesh.
 Left: coarse mesh. Right: fine mesh.
 }
 \label{fig:ds}
 \includegraphics[width=.45\textwidth]{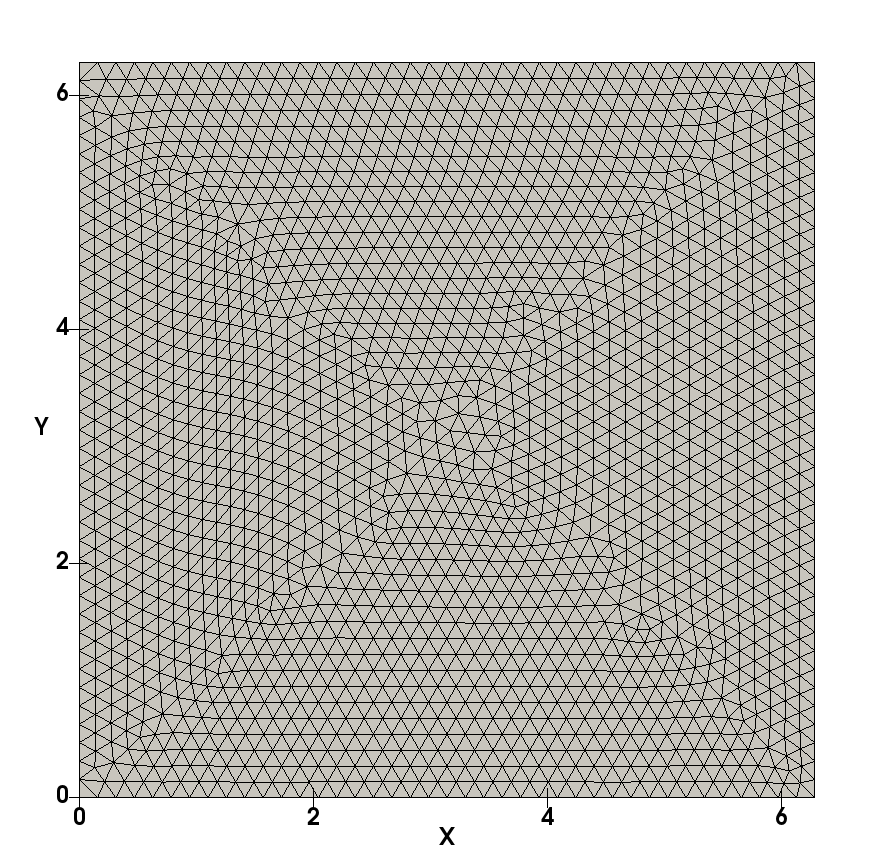}
 \includegraphics[width=.45\textwidth]{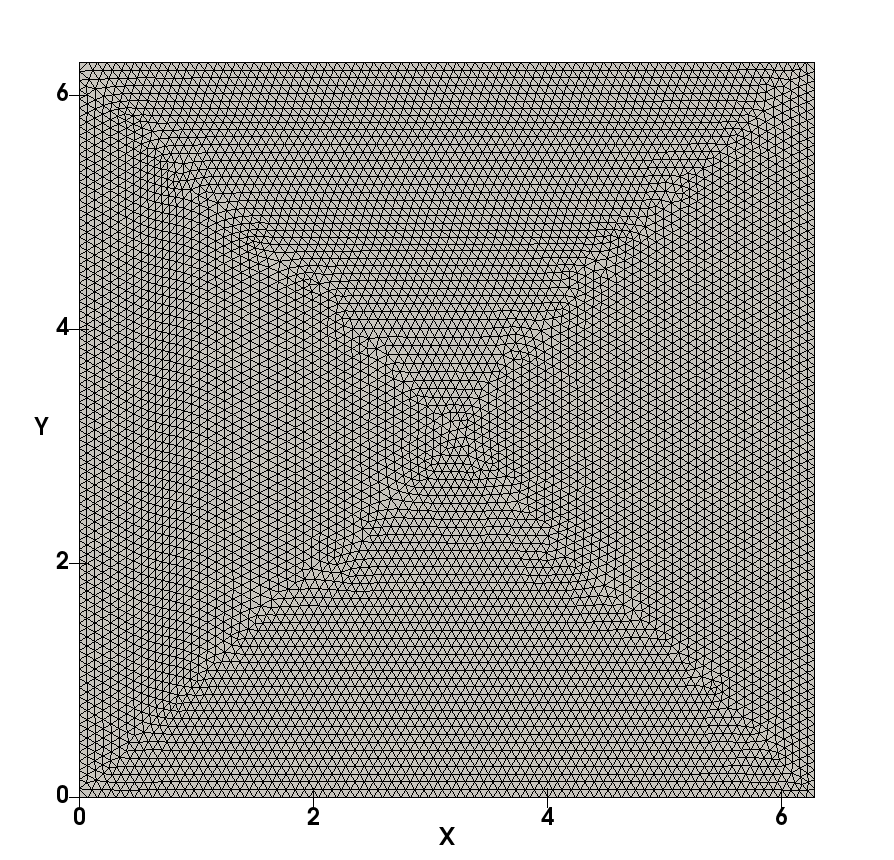}
\end{figure}

\begin{figure}[ht!]
 \caption{Example 2: the time history of energy and enstrophy.
 }
 \label{fig:ds0}
 \begin{tikzpicture} 
\begin{axis}[
	width=1\textwidth,
	height=0.25\textheight,
   	ylabel={$||u_h(t) ||_0^2$},
   	yticklabel style={/pgf/number format/fixed,/pgf/number format/precision=4},   	
   	every axis plot/.append style={line width=2pt, smooth},
   	no markers,   	
   	legend style={at={(0.02,0.02)},anchor=south west}
	]
\addplot table[x index=0, y index=1]{dsp3c.csv};
\addplot table[x index=0, y index=1]{dsp3f.csv};
\addlegendentry{$P^3, h = 2\pi/40$}
\addlegendentry{$P^3, h = 2\pi/80$}
\end{axis}
\end{tikzpicture}

\begin{tikzpicture} 
\begin{axis}[
	width=1\textwidth,
	height=0.25\textheight,
	xlabel={time},
   	ylabel={$||\omega_h(t) ||_0^2$},
   	every axis plot/.append style={line width=2pt, smooth},
   	no markers,   	
   	legend style={at={(0.02,0.02)},anchor=south west}
	]
\addplot table[x index=0, y index=2]{dsp3c.csv};
\addplot table[x index=0, y index=2]{dsp3f.csv};
\addlegendentry{$P^3, h = 2\pi/40$}
\addlegendentry{$P^3, h = 2\pi/80$}
\end{axis}
\end{tikzpicture}
\end{figure}

\begin{figure}[ht!]
 \caption{Example 2: Contour of vorticity.
 30 equally spaced contour lines between $-4.9$ and $4.9$.
 Left: results on the coarse mesh; right: results on the fine mesh.
 Top: $t=6$; bottom: $t=8$.
 $P^3$ approximation.
}
 \label{fig:ds1}
 \includegraphics[width=.48\textwidth, height=0.48\textwidth]{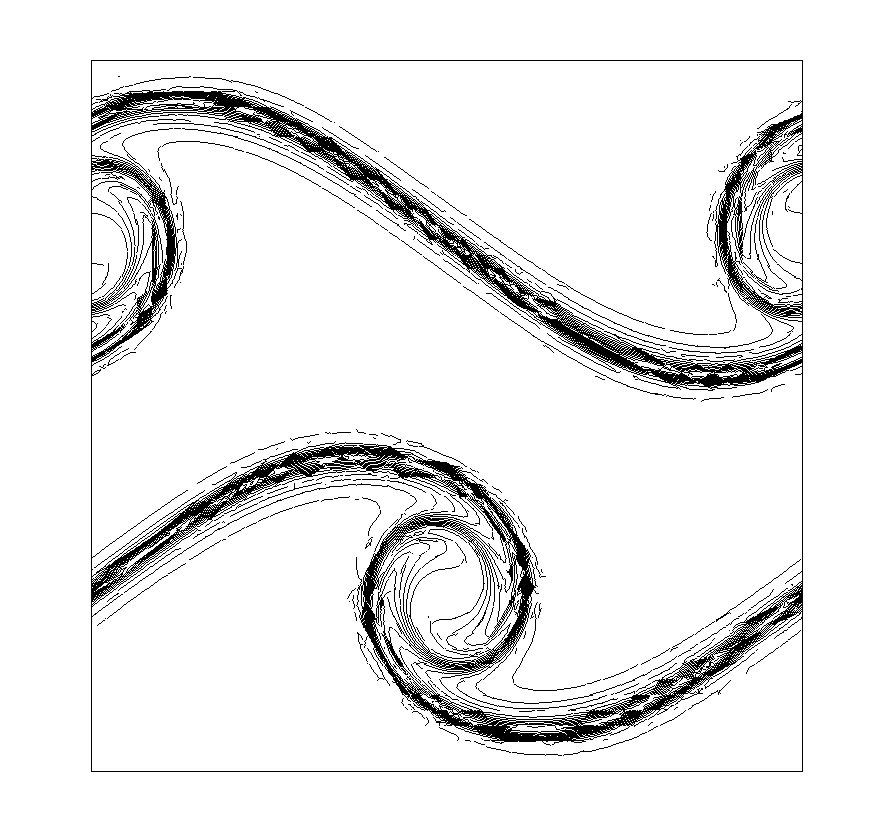}
 \includegraphics[width=.48\textwidth, height=0.48\textwidth]{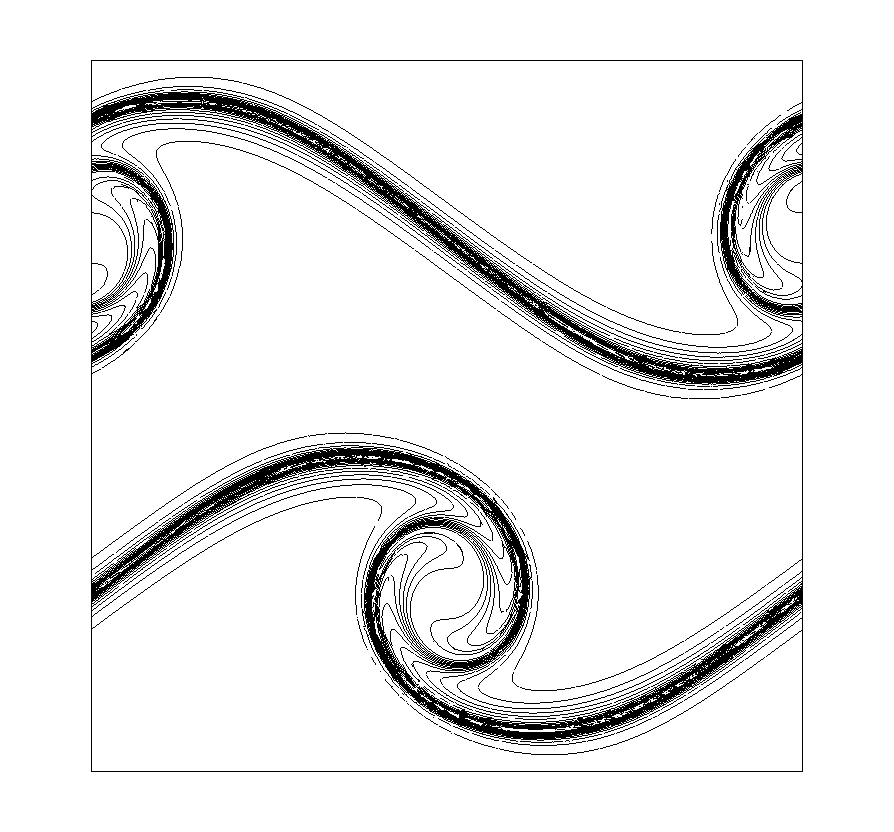}\\[.2ex]
 \includegraphics[width=.48\textwidth, height=0.48\textwidth]{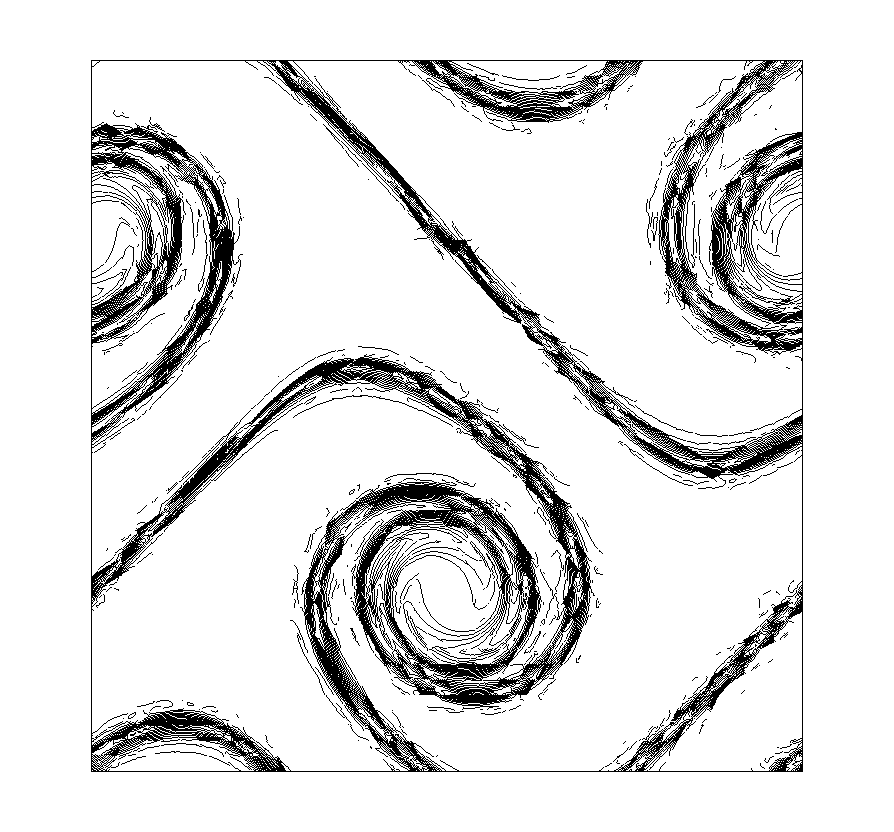}
 \includegraphics[width=.48\textwidth, height=0.48\textwidth]{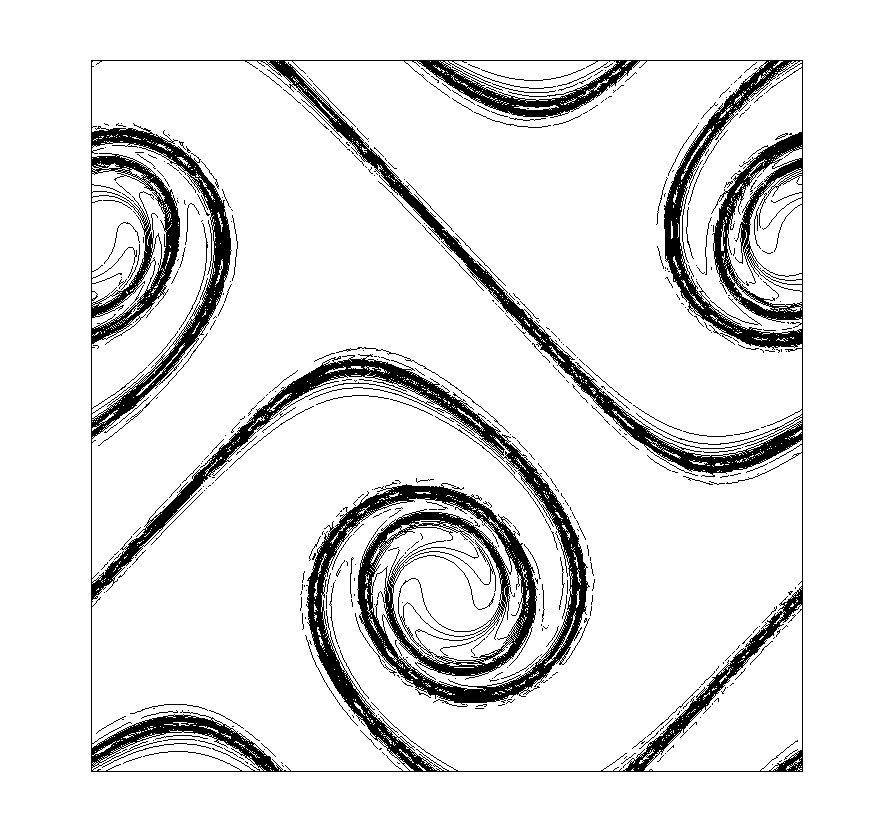}\\
\end{figure}

\subsection*{Example 3: \gfu{Kelvin}-Helmholtz instability problem}
We consider the \gfu{Kelvin}-Helmholtz instability problem, the set-up is
taken from \cite{Lehrenfeld18}.
The Navier-Stokes equations \eqref{ns} with Reynolds number $\mathrm{Re}=10000$ 
on the domain $[0,1]\times [0,1]$ with a periodic boundary condition on the $x$-direction, and 
the free-slip boundary condition 
$\bld u\cdot \bld n=0$, $\nu((\nabla\bld u)\bld n)\times\bld n=0$
at $y=0$ and $y=1$.
The initial conditions are taken to be 
\begin{align*}
 u_1(x,y,0) =&\; u_{\infty}\mathrm{tanh((2y-1)/\delta_0)}
 +c_n\partial_y \psi(x,y),\\
 u_2(x,y,0) =&\; -c_n\partial_x \psi(x,y),
\end{align*}
with corresponding stream function
\[
 \psi(x,y) = u_\infty\exp\left(-(y-.5)^2/\delta_0^2\right)[\cos(5\pi x)+\cos(20\pi x)].
\]
Here, $\delta_0=1/28$ denotes the vorticity thickness, $u_\infty=1$ is a reference velocity
and $c_n = 10^{-3}$ is a scaling/noise factor.
The scaled time $\bar t =\delta_0/u_\infty t$ is introduced.

We use a $P^4$ scheme \eqref{scheme-ns} with TVD-RK3 time stepping on an unstructured 
triangular mesh with mesh size $h=1/80$. The time step size is taken to be 
$\Delta t = \delta_0\times 10^{-2}$.
We run the simulation till time $\bar t = 400$ (a total of $40,000$ time steps).
The computation is performed on a desktop machine with 2 dual core CPUs, and 
about $20$ hours wall clock time is used for the overall simulation.

The time evolution of vortices
are shown in Fig.~\ref{fig:kh} up to time $\bar t=200$.
In the first row, the transition
from the initial condition to the four primary vortices is shown. 
The four vortices are unstable in the sense that they have the tendency to merge. This is a
well-known property of two-dimensional flows
for which energy is transferred from
the small to the large scales. 
We observe the second merging process is completed at around 
$\bar t = 56$, while the last merging process completed around 
$\bar t = 160$, and at time $\bar t=200$ a single vortex is left.
Comparing with the reference data \cite{Lehrenfeld18}, computed using
an IMEX SBDF2, 
 $P^8$ divergence-conforming HDG scheme \gfu{\cite{Lehrenfeld18}} on 
a $256\times 256$ uniform square mesh with time step 
size $\Delta t = \delta_0\times 10^{-3} \approx
3.6\times 10^{-5}$, 
we observe quite a good agreement of the vorticity dynamics up to time $\bar t=56$ where 
the second merging process is completed.
However, the numerical results in \cite{Lehrenfeld18} show that 
the last merging appears in a much later time, at around $\bar t=250$.
The numerical dissipation in our simulation 
triggered the last vortex merging in a much earlier time,
since we use a lower order method on a coarser mesh compared with \cite{Lehrenfeld18}.
We notice that a numerical simulation at the scale of \cite{Lehrenfeld18} is 
out of reach for our desktop-based simulation.

\begin{figure}[ht!]
 \caption{Example 3: Contour of vorticity $\omega_h:=\nabla_h\times \bld u_h$ 
 at (from left to right and top to bottom) time $\bar t=\{5,10,17,34,56,80,120,160,200\}$.
}
 \label{fig:kh}
 \includegraphics[width=.32\textwidth]{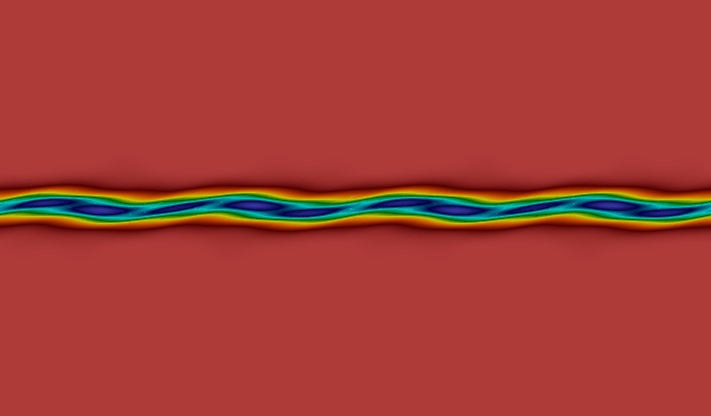}
 \includegraphics[width=.32\textwidth]{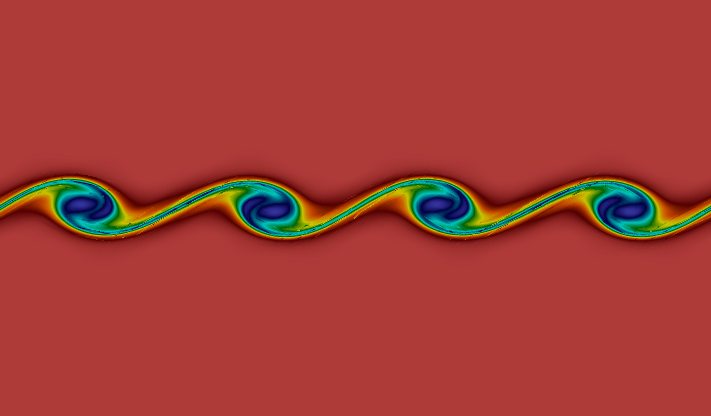}
 \includegraphics[width=.32\textwidth]{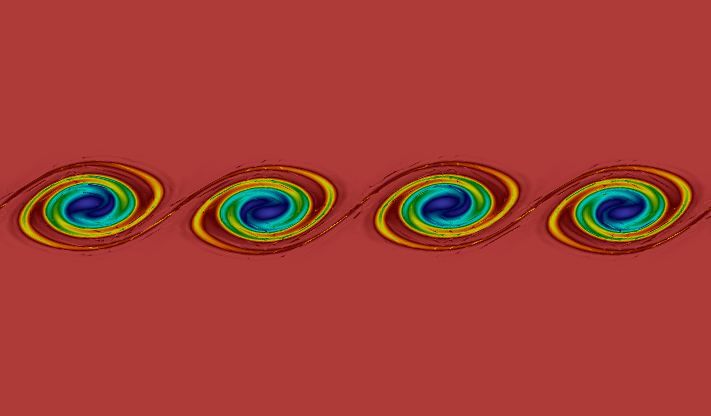}\\[.3ex]
 \includegraphics[width=.32\textwidth]{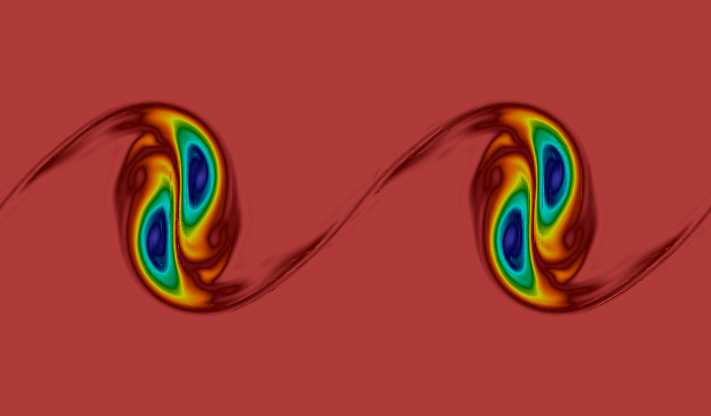}
 \includegraphics[width=.32\textwidth]{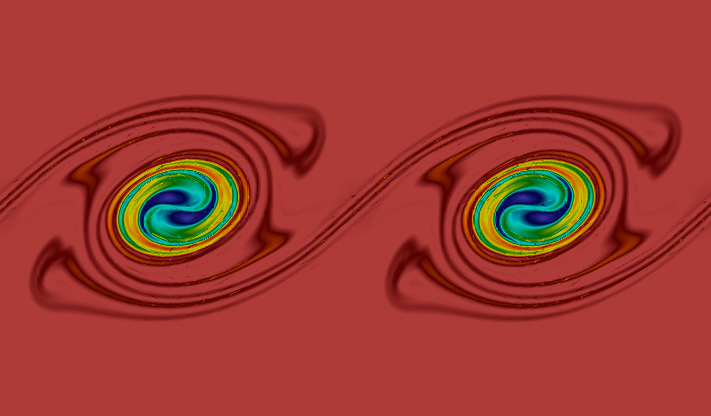}
 \includegraphics[width=.32\textwidth]{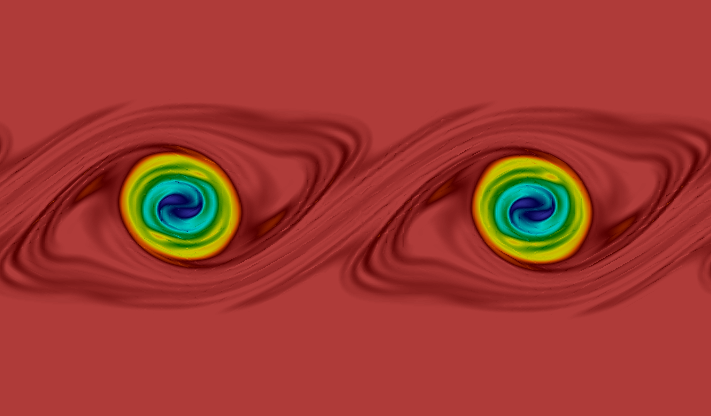}\\[.3ex]
 \includegraphics[width=.32\textwidth]{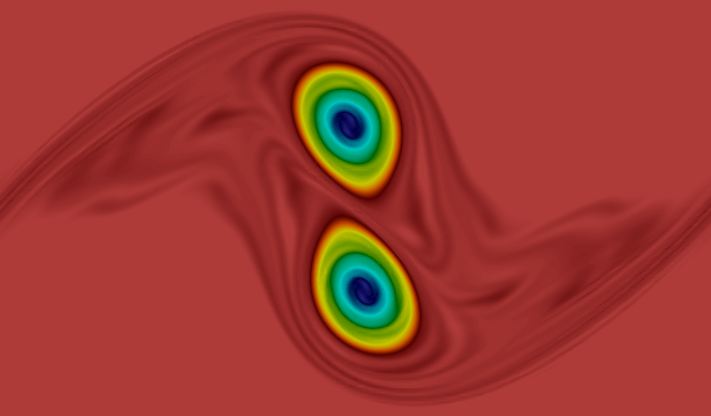}
 \includegraphics[width=.32\textwidth]{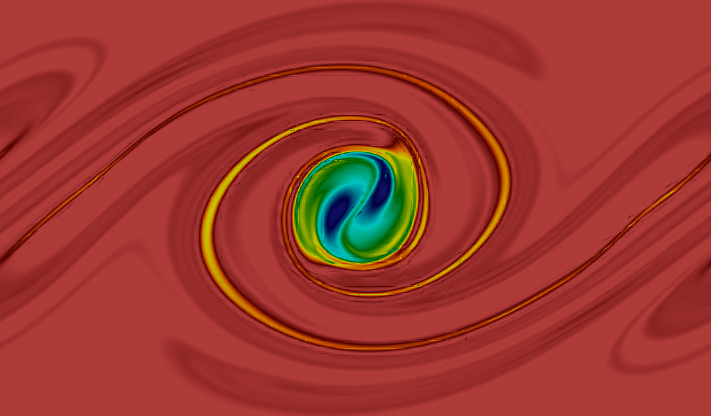}
 \includegraphics[width=.32\textwidth]{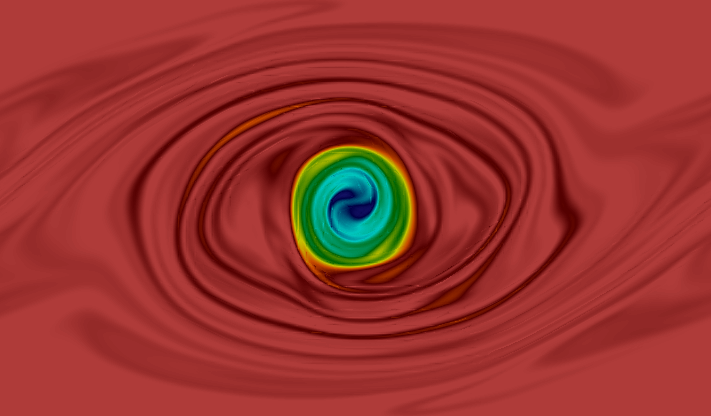}\\[.3ex]
 \includegraphics[width=.6\textwidth]{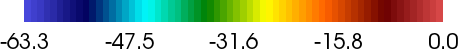}
\end{figure}

In Fig.~\ref{fig:kh2}, we plot the evolution of kinetic energy and enstrophy of our simulation, 
together with the reference data provided in \cite{Lehrenfeld18}.
A good agreement of the kinetic energy can be clearly seen, while the enstrophy
agrees pretty well till time $\bar t=150$, where the last vortex merging toke place for our 
simulation, while that happens at a much later time $\bar t=250$ for the 
scheme used in \cite{Lehrenfeld18}.
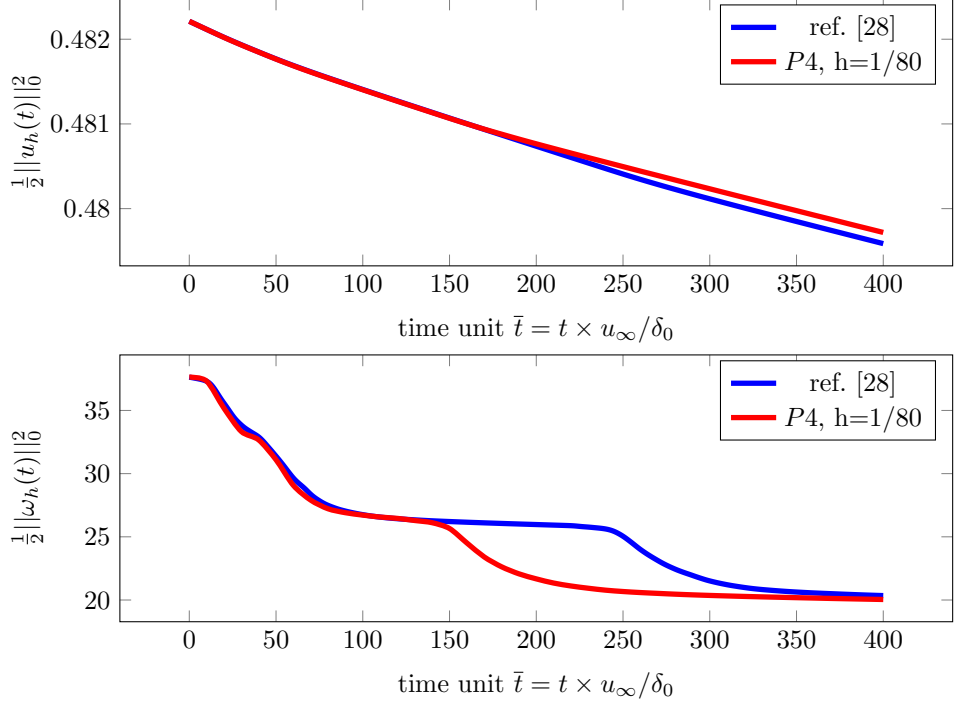
\begin{figure}
 \caption{Example 3: 
 the time history of energy and enstrophy.
}
 \label{fig:kh2}
\begin{tikzpicture} 
\begin{axis}[
	width=1\textwidth,
	height=0.25\textheight,
	xlabel={time unit $\overline{t}=t \times u_\infty / \delta_0$},
   	ylabel={$\frac{1}{2} ||u_h(t) ||_0^2$},
   	yticklabel style={/pgf/number format/fixed,/pgf/number format/precision=4},   	
   	every axis plot/.append style={line width=2pt, smooth},
   	no markers,   	
	]
\addplot table[x index=1, y index=2]{ref.csv};
\addplot table[x index=0, y index=1]{data.csv};
\addlegendentry{ref. \cite{Lehrenfeld18}}
\addlegendentry{$P4$, h=1/80}
\end{axis}
\end{tikzpicture}

\begin{tikzpicture} 
\begin{axis}[
	width=1\textwidth,
	height=0.25\textheight,
	xlabel={time unit $\overline{t}=t \times u_\infty / \delta_0$},
   	ylabel={$\frac{1}{2} ||\omega_h(t) ||_0^2$},
   	every axis plot/.append style={line width=2pt, smooth},
   	no markers,   	
	]
\addplot table[x index=1, y index=3]{ref.csv};
\addplot table[x index=0, y index=2]{data.csv};
\addlegendentry{ref. \cite{Lehrenfeld18}}
\addlegendentry{$P4$, h=1/80}
\end{axis}
\end{tikzpicture}
\end{figure}

\subsection*{Example 4: flow around a cylinder}
We consider the 2D-2 benchmark problem proposed in \cite{LD96} where a laminar flow around a 
cylinder is considered. 
The domain is a rectangular channel without an almost vertically centered circular obstacle, 
c.f. Fig.~\ref{fig:obs},
\[
 \Omega:=[0,2.2]\times [0,0.41]\backslash \{\|(x,y)-(0.2,0.2)\|_2\le 0.05\}.
\]
The boundary is decomposed into $\Gamma_{{in}}:=\{x=0\}$, the inflow boundary, 
$\Gamma_{{out}}:=\{x=2.2\}$, the outflow boundary, and 
$\Gamma_{{wall}}:=\partial\Omega\backslash(\Gamma_{{in}}\cup \Gamma_{{out}})$,
the wall boundary. On $\Gamma_{{out}}$ we prescribe 
natural boundary conditions
$(-\nu\nabla\bld u + p I) \bld n = 0$, on 
$\Gamma_{{wall}}$ homogeneous Dirichlet boundary conditions for the velocity (no-slip) 
and on $\Gamma_{{in}}$  the
inflow Dirichlet boundary conditions
\[\bld u(0, y, t) = 6\bar u\, y(0.41 - y)/0.41^2 \cdot (1, 0),
\]
with $\bar u=1$ the average inflow velocity.
The viscosity is taken to be $\nu = 10^{-3}$, hence Reynolds number $\mathrm{Re}=\bar u D/\nu = 100$,
where $D=0.1$ is the disc diameter.

The quantities of interest in this example are the (maximal and minimal) drag and lift 
forces $cD$ , $cL$ that
act on the disc. These are defined as
\[
 [c_D, c_L] =\frac{1}{\bar u^2 r}\int_{\Gamma_o}(\nu\nabla \bld u-pI)\bld n \mathrm{ds},
\]
where $r=0.05$ is the radius of the obstacle, and $\Gamma_o$ denotes the surface of the obstacle.

We use a (curved) unstructured triangular grid around the disk. In Fig. \ref{fig:obs} the geometry, 
the mesh and a typical solution is depicted. The final time of the simulation is taken to be $t=8$.
The mesh consists of 488 triangular elements with mesh size 
$h\approx 0.013$  around the circle (24 uniformly spaced nodes on the circle), 
and $h\approx 0.08$ away from the circle. 
We run the simulation on this mesh 
using polynomial degree form $2$ to $4$.
The maximal/minimal drag and lift coefficients are lists in Table \ref{table:ld}, where
the local dofs refer to those for velocity and pressure, while the global 
dofs refer to those for the Lagrange multiplier on the mesh skeleton.
As a reference, we also show the results obtained by {\sf FEATFLOW} \cite{FE}
using a $Q^2/P^{1,\mathrm{disc}}$ quadrilaterial element.
Clearly we observe a rapid convergence as the polynomial degree increases.
Compared with the (low-order) results form the literature \cite{FE}, 
the same accuracy is achieved
with a lot less degrees of freedom.
Similar observation was also found 
for the the divergence-conforming IMEX-HDG scheme \cite{LehrenfeldSchoberl16}.

\begin{figure}[ht!]
 \caption{Example 4: Sketch of the mesh and the solution using 
 the $P^4$ scheme
 (color corresponding to velocity magnitude
 $\|\bld u\|_2$).}
 \label{fig:obs}
 \includegraphics[width=.75\textwidth]{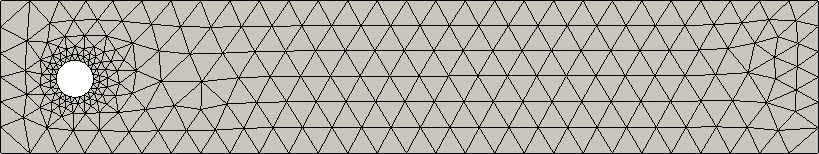}\\[.3ex]
 \includegraphics[width=.75\textwidth]{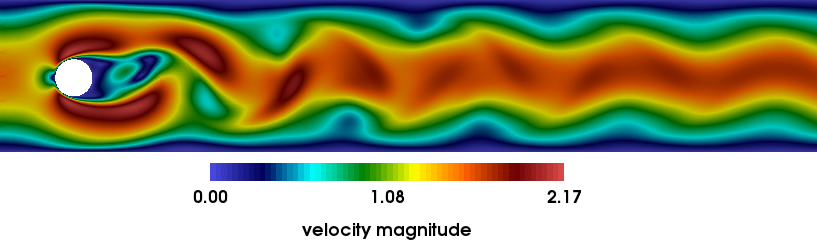}
\end{figure}

\begin{table}[ht!]
\caption{Example 4: Maximal/minimal values of lift and drag coefficients: 
results for different polynomial degree.} 
\centering 
{%
\begin{tabular}{ crr cccc}
\toprule
  &\begin{tabular}{c}
    $\#$dof\\
    local
   \end{tabular} & 
   \begin{tabular}{c}
    $\#$dof\\
    global
   \end{tabular}& 
  $\max c_D$
& 
  $\min c_D$
  & 
  $\max c_L$
  & 
  $\min c_L$
\tabularnewline
\midrule
{k=2} 
& 5 368& 2 316& 
3.132939 & 3.074858 & 0.935284 & -0.884771
\\ [1.5ex]
{k=3} 
&7 808& 3 088  &
3.229686 & 3.170424 & 0.969323 & -0.965982
\\ [1.5ex]
{k=4} 
&10 736& 3 860  &
3.226865 & 3.163545 & 0.986497 & -1.018691
\\ [1.5ex]
\midrule
\multirow{2}{1mm}{ref. \cite{FE}} 
&-& 167 232
&3.22662& 3.16351& 0.98620& -1.02093\\
&-&667 264& 3.22711& 3.16426& 0.98658& -1.02129
\\
\bottomrule
\end{tabular}}
\label{table:ld} 
\end{table}

\subsection*{Example 5: lid driven cavity at a high Reynolds number}
In our last example, we consider a lid driven cavity flow problem 
\cite{Ghia82} at a high Reynolds number $\mathrm{Re}=10,000$.
The domain is a unit square $\Omega = [0,1]\times [0,1]$, the velocity boundary 
condition is used on the boundary with $(u_1,u_2)=(1,0)$ on the top boundary $y=1$,
and $(u_1,u_2)=(0,0)$ on the other boundaries.
We use a steady-state Solver solver to generate the initial condition.
For this problem, the solution eventually reaches a steady state.
However, for such high Reynolds number flow, the numerical solution 
tend to oscillate around the steady-state without settling done 
on coarse meshes, see e.g. the discussion in \cite{Erturk09}.

We consider a uniform $32\times 32$ rectangular mesh using the 
divergence-free RT $Q^4$ velocity space. Since temporal accuracy is not of concern for this 
problem. We use the cheaper forward Euler time stepping \eqref{FE}.
The time step size is taken to be $\Delta t = 10^{-3}$.  Final time of simulation is $t=400$.
Hence, a total of $400,000$ time steps is used. 
For this problem, we have local dofs $37,888$(local velocity and pressure) 
and global dofs $10,560$(Lagrange multiplier).
The overall wall computational time is about 10 hours.

In Fig.~\ref{fig:ld} and Fig.~\ref{fig:ld2}, we plot the time evolution of the streamlines 
and vorticity contours. We numerically observe that starting around time 
$t=80$, the solution oscillates around the steady-state solution but never 
reaches the steady state. The $L^2$-norm of the velocity difference at two 
consecutive time levels hangs 
at around $5\times 10^{-5}$ and never drops down.
This phenomenon is probably due to the low mesh resolution ($32\times 32$ in our case).
In particular, for second-order methods, a mesh larger than $256\times 256$ 
shall be used to reach a steady-state for high Reynolds number flow ($Re>10, 000$);
see \cite{Erturk09}.
%
However, the main features of the small structure around the top, 
left and right corners can be  be clearly seen in 
Fig.~\ref{fig:ld} and Fig.~\ref{fig:ld2} starting at time $t=80$.
Finally, in Fig.~\ref{fig:ld0}, we plot the $x$-component of velocity field along the horizontal central line $x=0.5$, 
and the $y$-component of velocity field along the vertical central line $y=0.5$ at time $t=160,200,400$, along with the 
reference data provided in \cite{Ghia82}. A good match with the reference data is observed.

\begin{figure}[ht!]
 \caption{Example 5: Streamline plots at (form left to right and top to bottom)
time  $t=\{2,4,10,20,40,80,160,200,400\}$.}
 \label{fig:ld}
 \includegraphics[width=.32\textwidth, height=0.32\textwidth]{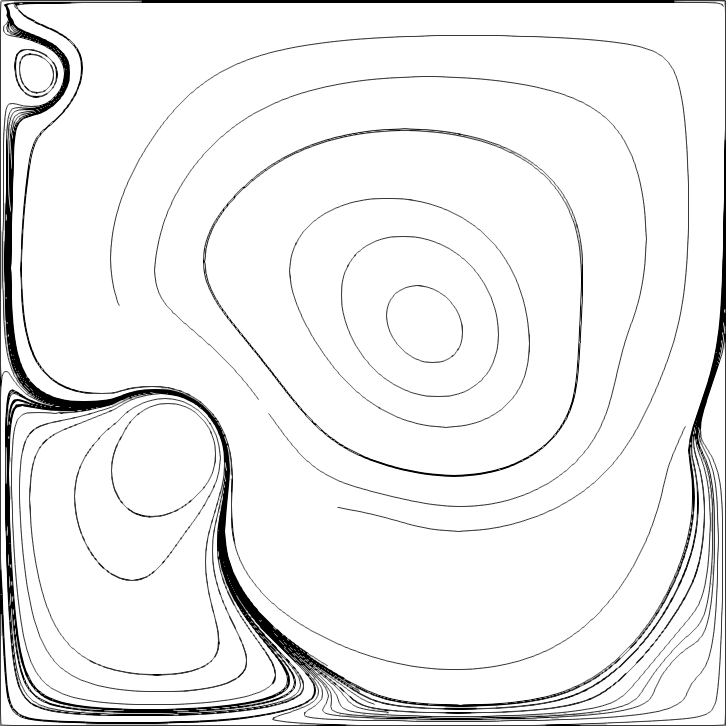}\hspace{.3ex}
 \includegraphics[width=.32\textwidth, height=0.32\textwidth]{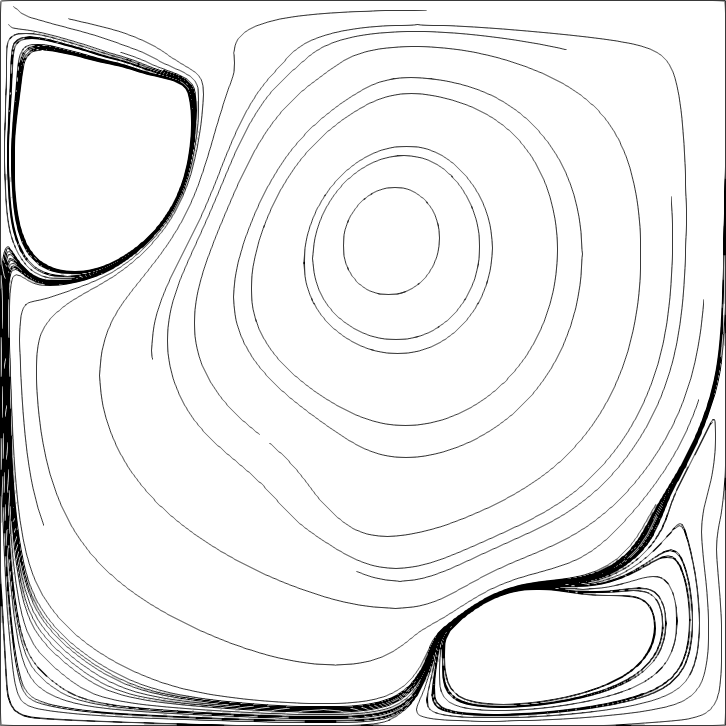}\hspace{.3ex}
 \includegraphics[width=.32\textwidth, height=0.32\textwidth]{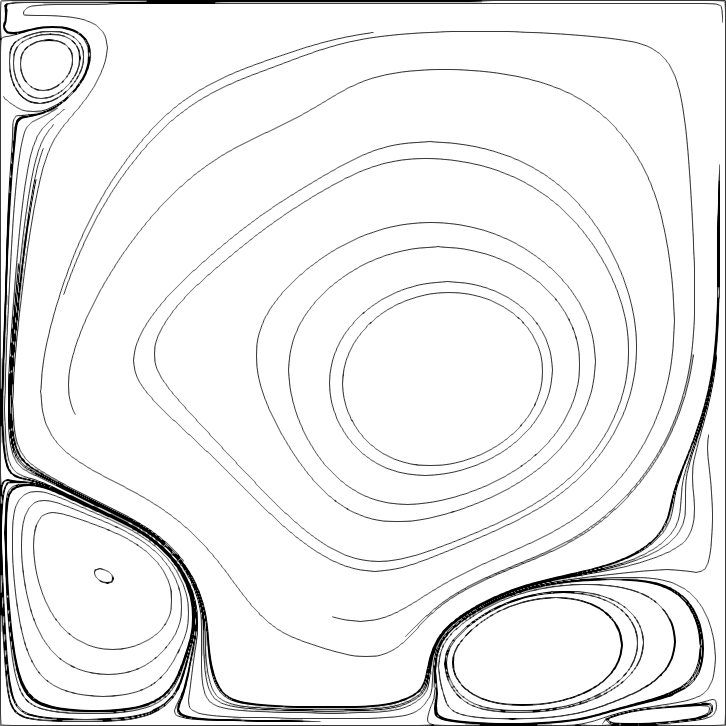}\\[.6ex]
 \includegraphics[width=.32\textwidth, height=0.32\textwidth]{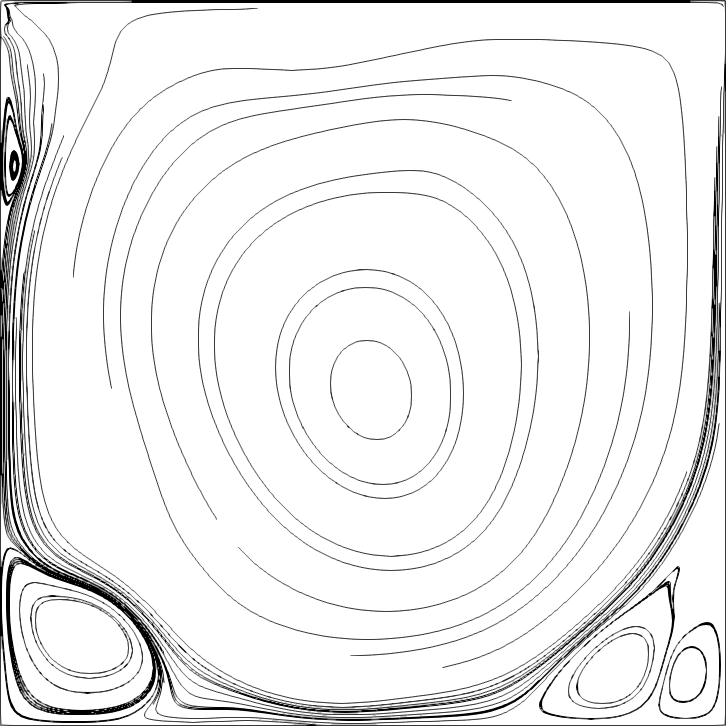}\hspace{.3ex}
 \includegraphics[width=.32\textwidth, height=0.32\textwidth]{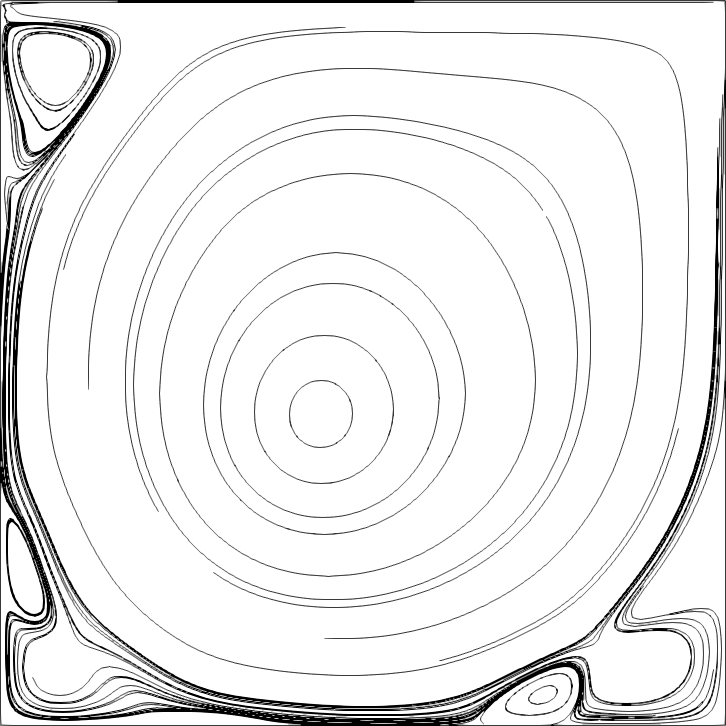}\hspace{.3ex}
 \includegraphics[width=.32\textwidth, height=0.32\textwidth]{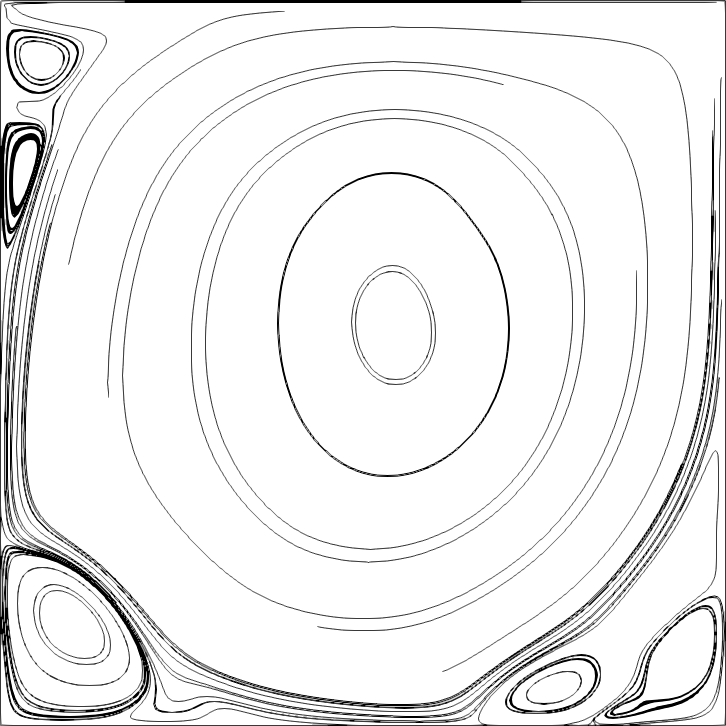}\\[.6ex]
 \includegraphics[width=.32\textwidth, height=0.32\textwidth]{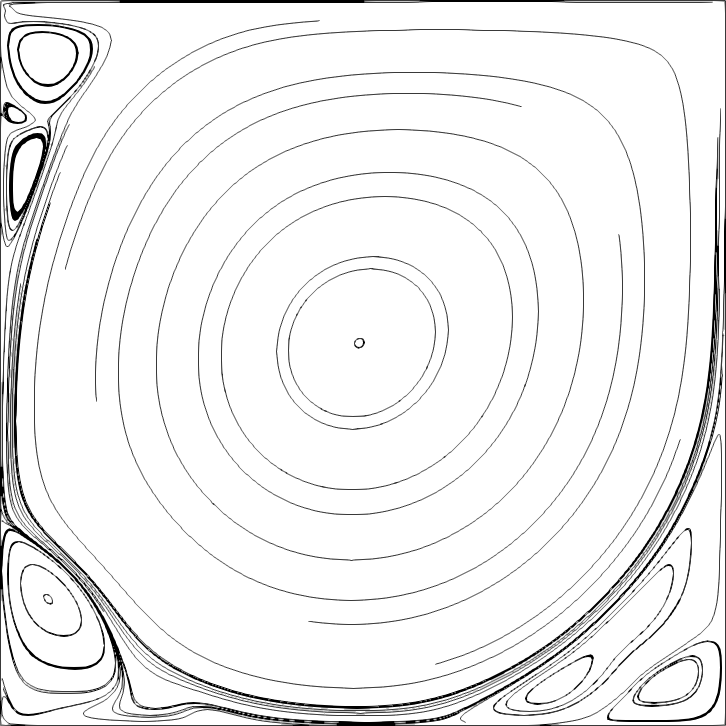}\hspace{.3ex}
 \includegraphics[width=.32\textwidth, height=0.32\textwidth]{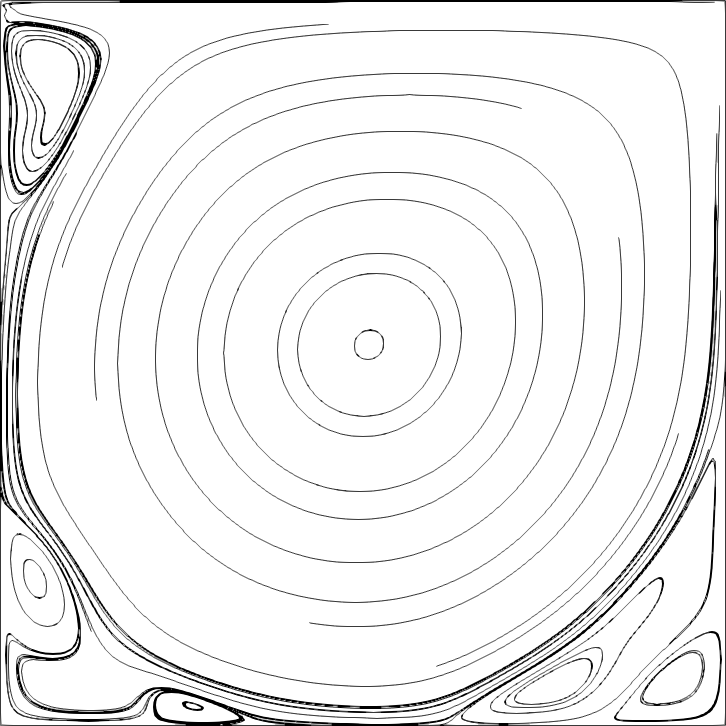}\hspace{.3ex}
 \includegraphics[width=.32\textwidth, height=0.32\textwidth]{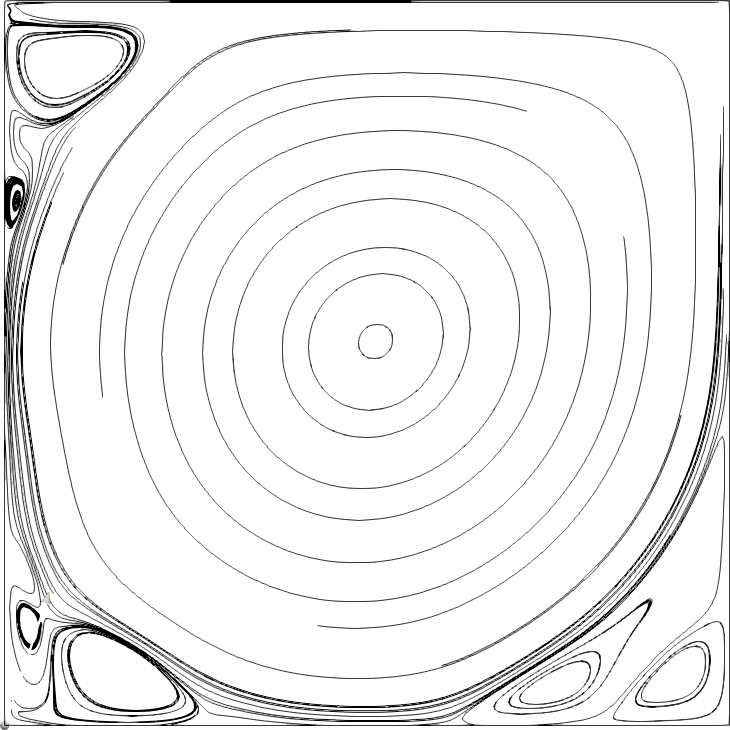}
\end{figure}

\begin{figure}[ht!]
 \caption{Example 5: Contour of vorticity
 at (form left to right and top to bottom) time $t=\{2,4,10,20,40,80,160,200,400\}$.
 30 equally spaced contours between $-1$ to $1$.
}
 \label{fig:ld2}
 \includegraphics[width=.32\textwidth, height=0.32\textwidth]{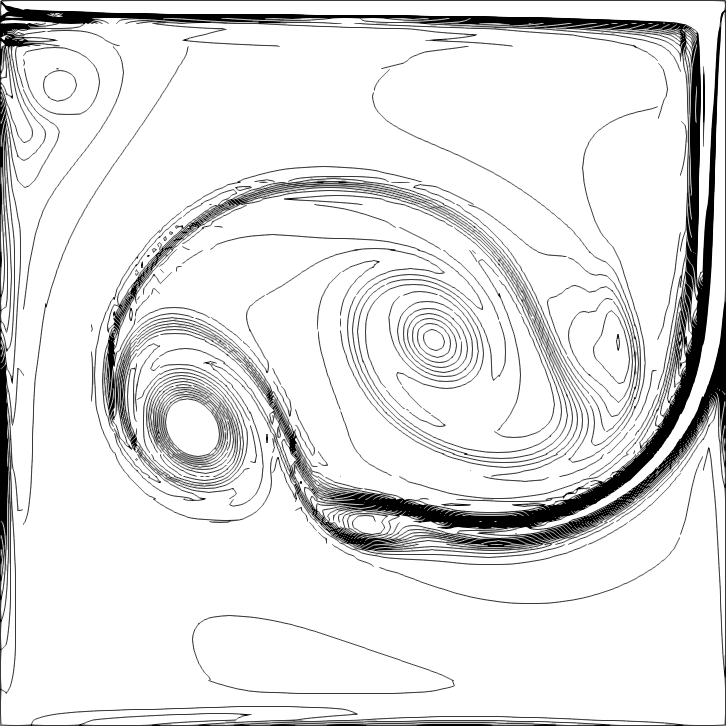}\hspace{.3ex}
 \includegraphics[width=.32\textwidth, height=0.32\textwidth]{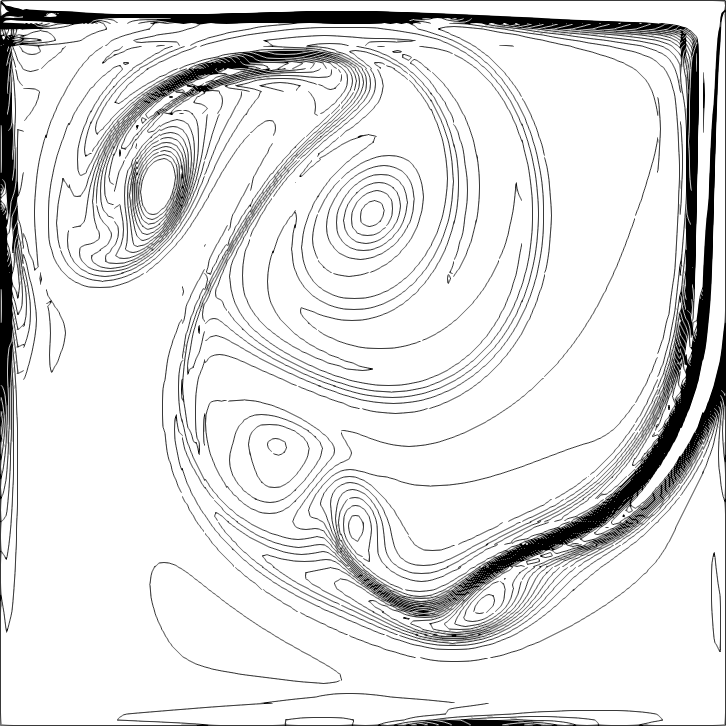}\hspace{.3ex}
 \includegraphics[width=.32\textwidth, height=0.32\textwidth]{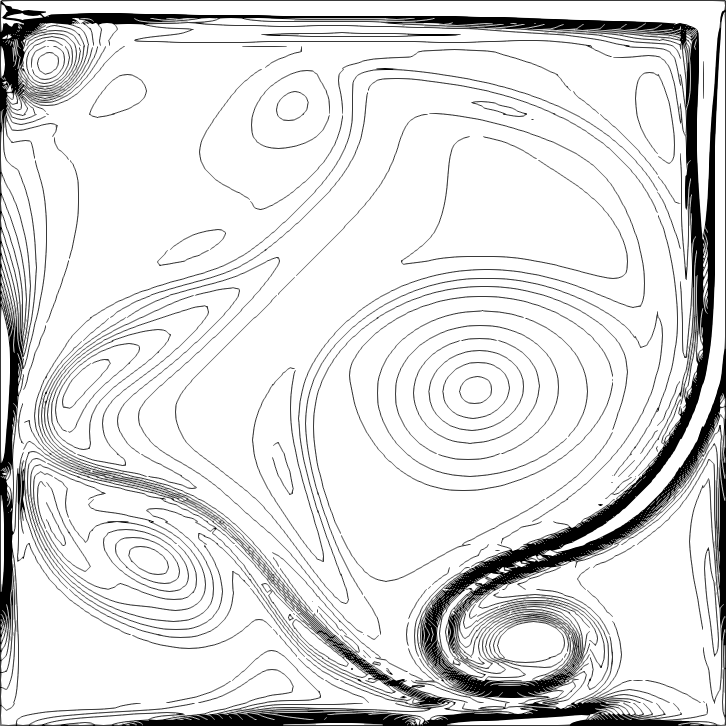}\\[.6ex]
 \includegraphics[width=.32\textwidth, height=0.32\textwidth]{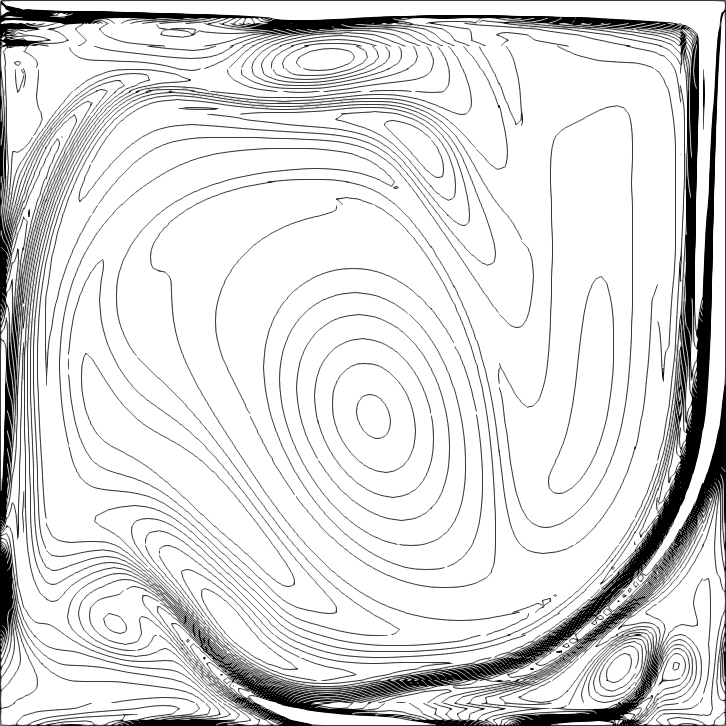}\hspace{.3ex}
 \includegraphics[width=.32\textwidth, height=0.32\textwidth]{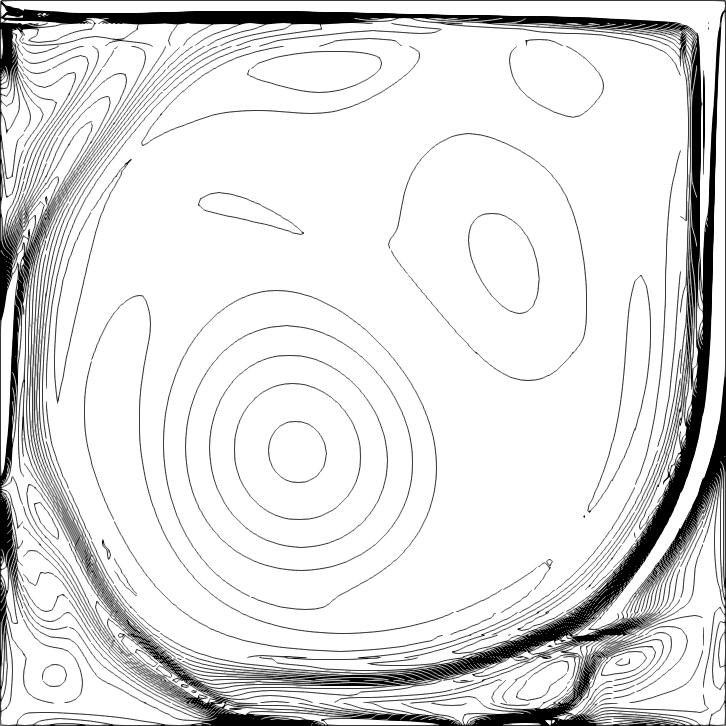}\hspace{.3ex}
 \includegraphics[width=.32\textwidth, height=0.32\textwidth]{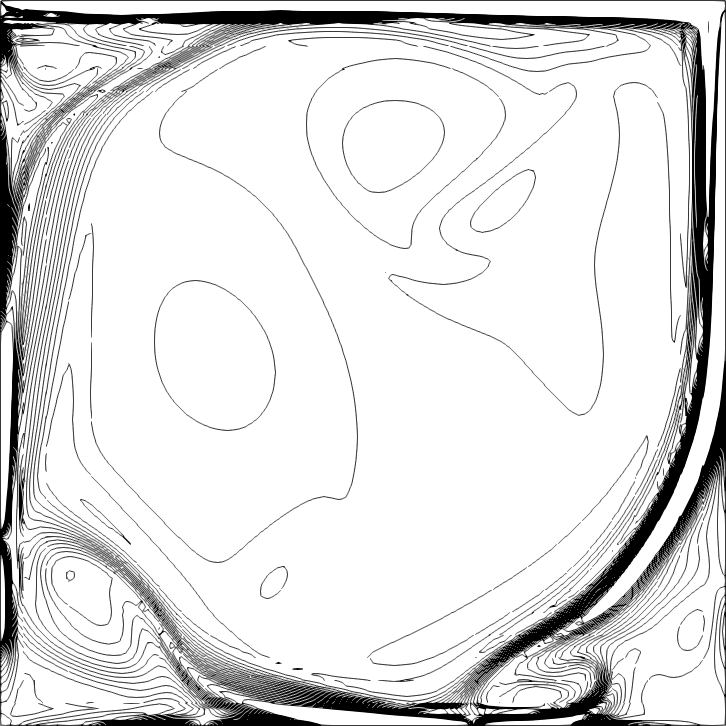}\\[.6ex]
 \includegraphics[width=.32\textwidth, height=0.32\textwidth]{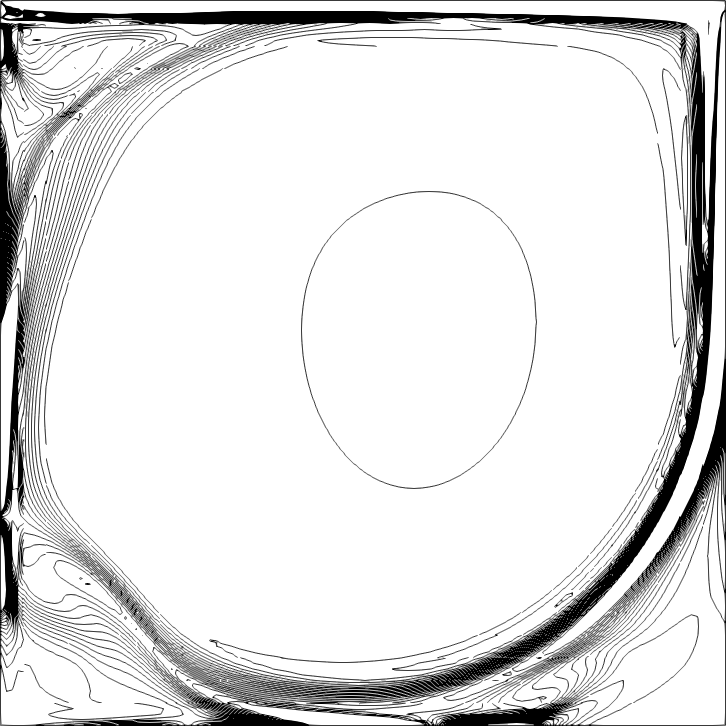}\hspace{.3ex}
 \includegraphics[width=.32\textwidth, height=0.32\textwidth]{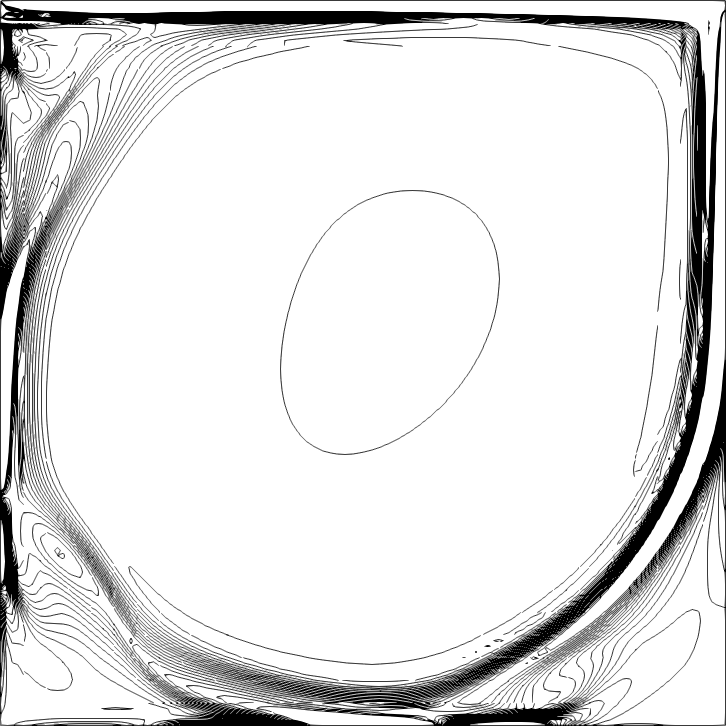}\hspace{.3ex}
 \includegraphics[width=.32\textwidth, height=0.32\textwidth]{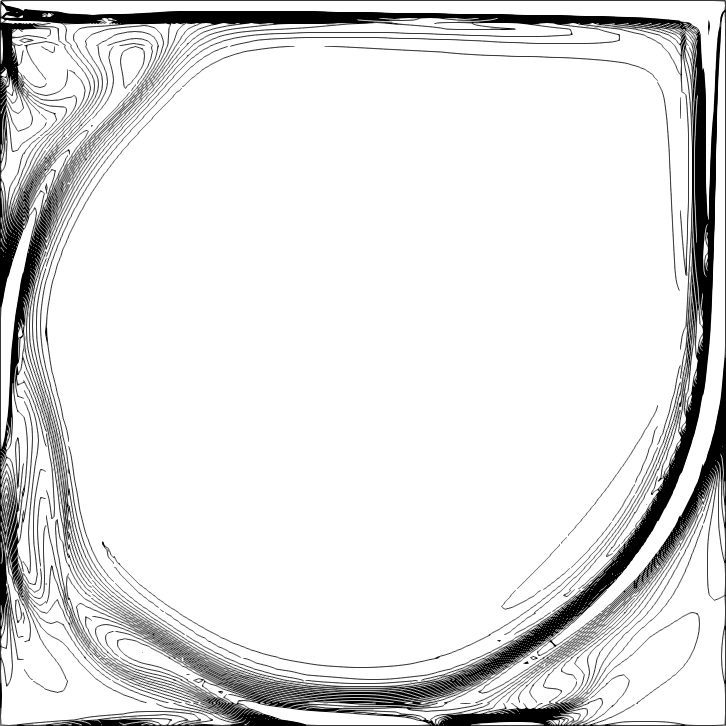}
\end{figure}
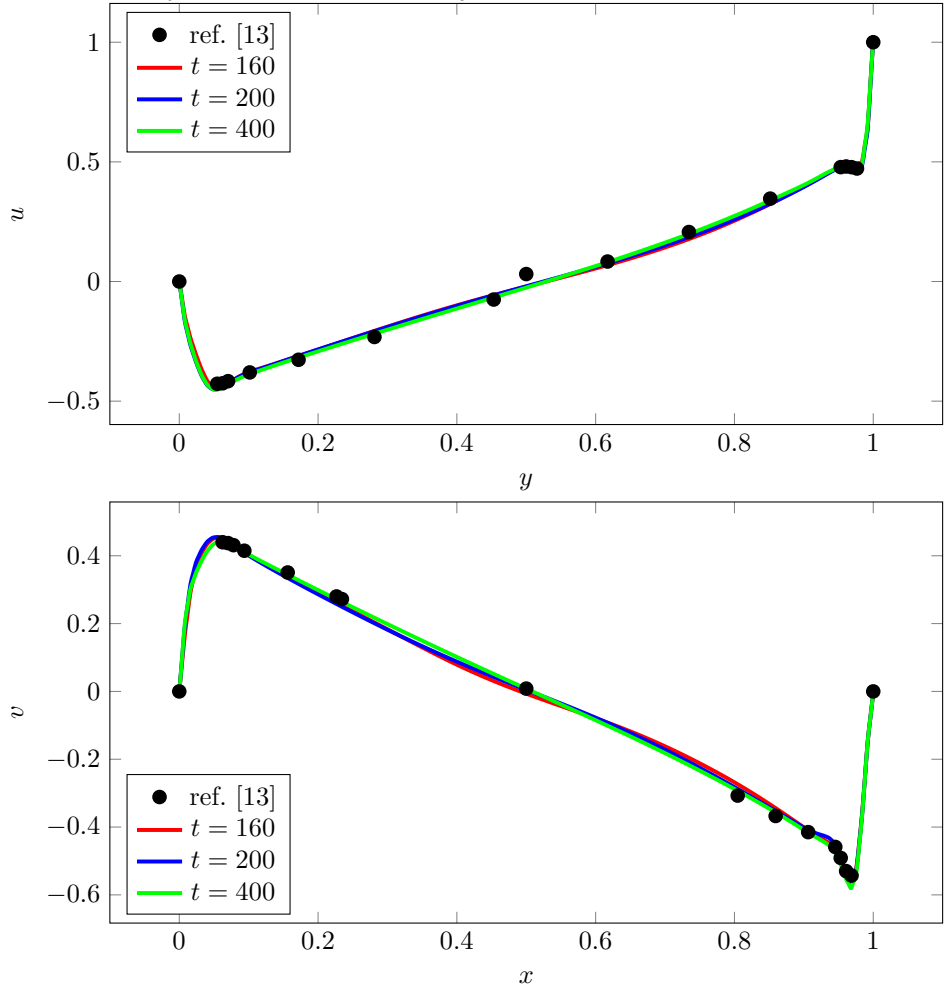
\begin{figure}[ht!]
 \caption{Example 5: velocity along cut lines. 
Top: $x$-component velocity on horizontal central line $x=0.5$;
bottom: $y$-component velocity on vertical central line $y=0.5$.
 }
 \label{fig:ld0}
 \begin{tikzpicture} 
\begin{axis}[
	width=1\textwidth,
	height=0.35\textheight,
 	xlabel={$y$},
   	ylabel={$u$},
   	yticklabel style={/pgf/number format/fixed,/pgf/number format/precision=4},   	
   	every axis plot/.append style={line width=1.5pt, smooth},
   	no markers,   	
   	legend style={at={(0.02,0.98)},anchor=north west}
	]
\addplot[color=black,solid,mark=*,mark options={solid},only marks] table[x index=0, y index=1]{ref_u.csv};
\addplot table[x index=0, y index=1]{u160x.csv};
\addplot[color=blue] table[x index=0, y index=1]{u200x.csv};
\addplot[color=green] table[x index=0, y index=1]{u400x.csv};
 \addlegendentry{ref. \cite{Ghia82}}
 \addlegendentry{$t=160$}
 \addlegendentry{$t=200$}
 \addlegendentry{$t=400$}
\end{axis}
\end{tikzpicture}

 \begin{tikzpicture} 
\begin{axis}[
	width=1\textwidth,
	height=0.35\textheight,
 	xlabel={$x$},
   	ylabel={$v$},
   	yticklabel style={/pgf/number format/fixed,/pgf/number format/precision=4},   	
   	every axis plot/.append style={line width=1.5pt, smooth},
   	no markers,   	
   	legend style={at={(0.02,0.02)},anchor=south west}
	]
\addplot[color=black,solid,mark=*,mark options={solid},only marks] table[x index=0, y index=1]{ref_v.csv};
\addplot table[x index=0, y index=1]{v160x.csv};
\addplot[color=blue] table[x index=0, y index=1]{v200x.csv};
\addplot[color=green] table[x index=0, y index=1]{v400x.csv};
 \addlegendentry{ref. \cite{Ghia82}}
 \addlegendentry{$t=160$}
 \addlegendentry{$t=200$}
 \addlegendentry{$t=400$}
\end{axis}
\end{tikzpicture}
\end{figure}

\section{Conclusion}
\label{sec:conclude}
We presented an explicit divergence-free 
DG method for incompressible flows. The key ingredient for the efficient implementation 
is the identification of the equivalence of the mass matrix inversion of the divergence-free 
finite element space and a hybrid-mixed Poisson solver.
The scheme is especially suitable for unsteady inviscid flow or viscous flow at a high Reynolds number 
flow.
\bibliographystyle{siam}

\end{document}